\documentclass[onefignum,onetabnum]{siamart220329}
\allowdisplaybreaks

\usepackage{lipsum}
\usepackage[export]{adjustbox}
\usepackage{amsfonts}
\usepackage{graphicx}
\usepackage{epstopdf}
\usepackage{algorithmic}
\usepackage{cancel}
\usepackage{tikz}
\usepackage{pgfplots}
\usetikzlibrary{shapes.arrows, patterns, calc}
\usepackage{tikz-3dplot}
\ifpdf
  \DeclareGraphicsExtensions{.eps,.pdf,.png,.jpg}
\else
  \DeclareGraphicsExtensions{.eps}
\fi

\usepgfplotslibrary{groupplots}
\usetikzlibrary{arrows}
\usetikzlibrary{shapes}
\usetikzlibrary{decorations.text}
\usetikzlibrary{quantikz}
\usepackage{siunitx}
\usepackage{comment}
\usetikzlibrary{arrows.meta}
\definecolor{steelblue}{HTML}{A1BDC7}
\definecolor{orange}{HTML}{D98C21}
\definecolor{silver}{HTML}{B0ABA8}
\definecolor{rust}{HTML}{B8420F}
\definecolor{seagreen}{HTML}{2E6B69}
\definecolor{joshua}{HTML}{FBDC7F}
\definecolor{darksky}{HTML}{154c79}

\colorlet{lightsilver}{silver!30!white}
\colorlet{lightlightsilver}{silver!10!white}
\colorlet{darkorange}{orange!85!black}
\colorlet{darksilver}{silver!85!black}
\colorlet{darksteelblue}{steelblue!85!black}
\colorlet{darkrust}{rust!85!black}
\colorlet{darkseagreen}{seagreen!85!black}

\definecolor{gblue}{HTML}{1D428A}
\definecolor{bgreen}{HTML}{00471B} 
\definecolor{hred}{HTML}{C8102E}  
\definecolor{cgreen}{HTML}{007A33}  

\definecolor{fred}{HTML}{B40F20}
\definecolor{fdarkorange}{HTML}{E58606}

\definecolor{ztealblue}{HTML}{3B9AB2}   
\definecolor{zlightblue}{HTML}{78B7C5}  
\definecolor{zyellow}{HTML}{EBCC2A}     
\definecolor{zgolden}{HTML}{E1AF00}     
\definecolor{zred}{HTML}{F21A00}

\colorlet{DGColor}{joshua}
\colorlet{DNPPGColor}{orange}
\colorlet{DFRGColor}{rust}
\colorlet{ExactColor}{steelblue}

\pgfplotsset{ compat=1.18,
    standard/.style={
      scale only axis,
      width=0.5\textwidth,
      enlarge x limits=0.05,
      enlarge y limits=0.05,
      max space between ticks=40,
      every axis/.append style={font=\small},
      every legend/.append style={font=\small},
      every node/.append style={font=\small},    
      },
      3d/.style={
      colormap={darkskycolormap}{
        color(0.0)=(zred),  
        color(0.05)=(zyellow),
        color(0.5)=(cgreen),
        color(1.0)=(gblue),
      }
      },
}

\usepackage{mathtools}
\usepackage{bbm}
\usepackage{amsmath}
\usepackage{amssymb}
\usepackage{stackengine}
\usepackage{fixmath}
\usepackage{bm}
\mathtoolsset{centercolon}  % Makes := typeset correctly for definitions

\renewcommand{\phi}{\varphi}
\renewcommand{\epsilon}{\varepsilon}

\DeclareFontFamily{U}{mathx}{\hyphenchar\font45}
\DeclareFontShape{U}{mathx}{m}{n}{
      <5> <6> <7> <8> <9> <10>
      <10.95> <12> <14.4> <17.28> <20.74> <24.88>
      mathx10
      }{}
\DeclareSymbolFont{mathx}{U}{mathx}{m}{n}
\DeclareFontSubstitution{U}{mathx}{m}{n}
\DeclareMathAccent{\widecheck}{0}{mathx}{"71}
\DeclareMathAccent{\wideparen}{0}{mathx}{"75}

\newcommand{\cnst}[1]{\mathrm{#1}}

\newcommand{\zerovct}{\mathbf{0}} % Zero vector
\newcommand{\Id}{\mathbf{I}} % Identity matrix

\newcommand{\zeromtx}{\bm{0}}

\providecommand{\mathbbm}{\mathbb} % In case we don't load bbm

\newcommand{\R}{\mathbbm{R}}

\newcommand{\vct}[1]{\mathbold{#1}}
\newcommand{\mtx}[1]{\mathbold{#1}}
\newcommand{\tp}{{T}}

\newcommand{\diff}{\Phi}

\DeclarePairedDelimiterX{\infdivx}[2]{(}{)}{%
  #1\;\delimsize\|\;#2%
}
\newcommand{\Div}[1]{\mathbb{D}_{#1}\infdivx}
\newcommand{\DivKL}{\Div{\mathrm{KL}}}

\DeclarePairedDelimiterX{\Biginfdivx}[2]{\Big(}{\Big)}{%
  #1\;\delimsize\big\|\;#2%
}
\newcommand{\BigDiv}[1]{\mathbb{D}_{#1}\Biginfdivx}
\newcommand{\BigDivKL}{\BigDiv{\mathrm{KL}}}

\DeclarePairedDelimiterX{\biginfdivx}[2]{\big(}{\big)}{%
  #1\;\delimsize\|\;#2%
}
\newcommand{\bigDiv}[1]{\mathbb{D}_{#1}\biginfdivx}
\newcommand{\bigDivKL}{\bigDiv{\mathrm{KL}}}

\newcommand{\divergence}{\operatorname{div}}
\newcommand\D{\mathop{}\cnst{d}}
\newcommand{\pdedomain}{\Omega}
\newcommand{\inbound}{\widecheck{\Gamma}} 
\newcommand{\outbound}{\widehat{\Gamma}} 
\newcommand{\indicator}{\mathbbm{1}}

\newcommand{\tr}{\phi}
\newcommand{\te}{\psi}
\newcommand{\trc}{r}

\newcommand{\trf}{\hat{\rho}}
\newcommand{\tef}{\hat{\sigma}}
\newcommand{\trs}{V}

\newcommand\Ccancel[2][black]{
    \let\OldcancelColor\CancelColor
    \renewcommand\CancelColor{\color{#1}}
    \cancel{#2}
    \renewcommand\CancelColor{\OldcancelColor}
}

\usepackage{pgfplotstable} 
\usepackage{booktabs}
\usepackage[section]{placeins}
\usepackage{float}
\usepackage{placeins}

\ifpdf
  \DeclareGraphicsExtensions{.eps,.pdf,.png,.jpg}
\else
  \DeclareGraphicsExtensions{.eps}
\fi

\newsiamremark{remark}{Remark}
\newsiamremark{hypothesis}{Hypothesis}
\crefname{hypothesis}{Hypothesis}{Hypotheses}
\newsiamthm{claim}{Claim}

\newcommand{\newglobalsiamremark}[2]{
  \theoremstyle{plain}
  \theoremheaderfont{\normalfont\itshape}
  \theorembodyfont{\normalfont}
  \theoremseparator{.}
  \theoremsymbol{}
  \newtheorem{#1}{#2}
}
\newglobalsiamremark{example}{Numerical Example}
\crefname{example}{Example}{Examples}

\headers{Maximum likelihood discretization of transport}{Brook Eyob and Florian Sch{\"a}fer}

\title{Maximum likelihood discretization of the transport equation} %\thanks{Submitted to the editors DATE.}}
\author{Brook Eyob\thanks{Georgia Tech 
  (\email{brook@gatech.edu}).}
\and Florian Sch{\"a}fer\thanks{Georgia Tech 
  (\email{fts@gatech.edu}).}}

\hypersetup{colorlinks=true,linkcolor=darkrust,citecolor=darkseagreen,urlcolor=darksilver}

\ifpdf
\hypersetup{
  pdftitle={Maximum likelihood discretization of the transport equation},
  pdfauthor={Brook Eyob and Florian Sch{\"a}fer}
}

\begin{document}

\maketitle

% REQUIRED
\begin{abstract}
  The transport of positive quantities underlies countless physical processes, including fluid, gas, and plasma dynamics.
  Discretizing the associated partial differential equations with Galerkin methods can result in spurious nonpositivity of solutions.
  We observe that these methods amount to performing statistical inference using the method of moments (MoM) and that the loss of positivity arises from MoM's susceptibility to producing estimates inconsistent with the observed data.
  We overcome this problem by replacing MoM with maximum likelihood estimation, introducing \emph{maximum likelihood discretization} (MLD).
  In the continuous limit, MLD simplifies to the Fisher-Rao Galerkin (FRG) semidiscretization, which replaces the $L^2$ inner product in Galerkin projection with the Fisher-Rao metric of probability distributions.
  We show empirically that FRG preserves positivity. We prove rigorously that it yields error bounds in the Kullback-Leibler divergence.
\end{abstract}

% REQUIRED

% REQUIRED

\section{Introduction}
\subsection{Background:}
\subsubsection*{Transport of particle densities \nopunct} along prescribed velocity fields is ubiquitous in fluid and gas dynamics, for instance as part of the Euler, Navier-Stokes, MHD, Boltzmann, and Vlasov equations.
In its simplest form, the transport of the density $\rho_t$ along the velocity field $\vct{u}_t$ follows the conservation law $\dot{\rho}_t + \divergence (\rho_t \vct{u}_t) = 0$.

\subsubsection*{Numerical methods} Due to the importance and ubiquity of transport phenomena, numerous methods address their numerical simulation, including finite difference, finite volume, and (discontinuous) Galerkin methods.
Finite difference methods track the density values on a grid imposed on the physical domain, replacing derivatives with difference quotients \cite{sod1978survey,leveque2007finite}.
Finite volume methods instead track cell-averages over a tessellation of the physical domain, interacting through mass-fluxes across cell-interfaces \cite{leveque2002finite,barth2018finite}.
This ensures local conservation, which is critical for the convergence to the correct weak solution \cite{shi2018local,lax2005systems}.
The latter is also attainable by suitably chosen finite difference methods \cite{merriman2003understanding}.
Galerkin methods project the solution onto a finite dimensional subspace such as finite element spaces \cite{hughes1986newI,hughes1986newII,hughes1989new,quarteroni2009numerical}.
Discontinuous Galerkin methods combine the flexibility of Galerkin methods with the local conservation property of flux-based finite volume--type schemes \cite{cockburn2012discontinuous}.

\subsubsection*{Loss of positivity and relative accuracy} In applications like the transport of mass, probability, or energy densities, the transported quantity $\rho$ is necessarily positive.
The transport equation itself does not prohibit negative $\rho$-values, but downstream tasks like sampling, computing pressures or reaction rates break down in their presence.
% While a simple logarithmic parametrization $\partial_t \left(\log(\rho)\right) + \divergence(\log(\rho_t) \vct{u}_t) + \log \left(\rho_t\right) \divergence(\vct{u}_t) = 0$ ensures positivity at the cost of local conservation.
In the setting of discontinuous Galerkin methods, \cite{zhang2010positivity,xu2017bound,shu2018bound} introduce a limiter that scales the solution in each cell toward its mean.
Under a CFL-like condition, the positivity of the mean is ensured, enabling the limiter to guarantee positivity of the solution.
This modification avoids negative values but still loses relative accuracy in regions of near-zero density.
This can cause significant errors in derived quantities like reaction rates or pressures.

\subsection{This work} 

\subsubsection*{A statistician's lens on Galerkin methods} 
This work introduces a new perspective on Galerkin methods for transport and their loss of positivity.
It views the Galerkin basis as a parametric statistical model used to infer the unknown probability distribution $\rho_t$.
The advection of particles causes $\rho_t$ to leave the statistical model and to compute the next time step, the statistician needs to estimate the element of the model that best represents the advected particle distribution.
Standard Galerkin methods amount to using the method of moments (MoM) for solving this inference problem.
A well-known flaw of MoM is that it can return parameters that are unable to produce the observed data.
In Galerkin methods for transport, this amounts to the loss of positivity of $\rho_t$.

\subsubsection*{Maximum likelihood discretization} The method of maximum likelihood avoids many flaws of the MoM and always returns parameters under which the observed data has nonzero probability.
We replace the method of moments in standard Galerkin methods with the method of maximum likelihood.
This seemingly requires solving a non-quadratic optimization problem at each time step.
But in the limit of small time steps, it simplifies to the \emph{Fisher-Rao semidiscretization} that replaces the standard $L^2$ inner product with the $\rho^{-1}_{t}$-weighted one.
Solving it by standard time steppers has minimal computational overhead.

\subsubsection*{Properties of Maximum likelihood discretization} 
Continuous and discontinuous Fisher-Rao Galerkin methods satisfy global and local conservation, respectively.
Different from conventional Galerkin methods, this is \emph{not} conditional on the Galerkin space containing (piecewise) constant functions.
Fisher-Rao Galerkin methods use a different inner product at each time step.
This makes it challenging to analyze the resulting approximation error with respect to any single inner product.
However, we show that Fisher-Rao Galerkin methods retain error bounds in the Kullback-Leibler divergence sense, to the same convergence order as the standard Galerkin method.
We furthermore show empirically that FRG's errors match those standard Galerkin methods in the $L^2$ and $L^1$ norms, but are significantly lower in the Kullback-Leibler divergence sense.
Our numerical experiments show that FRG maintains the positivity of solutions.

\subsubsection*{Practical performance}
We validate our Fisher-Rao Galerkin (FRG) method on an extensive suite of numerical experiments in one and two-dimensional problems.
Throughout, FRG is able to preserve positivity in settings where the traditional Galerkin method fails.
It also shows lower relative error in regions of small density due to the Fisher-Rao metric emphasizing relative errors over absolute errors, unlike the $L^{2}$ metric.
This results from the parametrization-invariance of the Kullback-Leibler divergence.
FRG maintains similar convergence rates under grid refinement with respect to the $L^2$ and $L^{1}$ norm as the traditional Galerkin method.
It shows improved convergence rate and error with respect to the Kullback-Leibler (KL) divergence.
The Fisher-Rao Galerkin method may break down when using large time steps, although this is unlikely to be an issue in practice since the CFL condition would generally preclude large time step sizes.

\section{Maximum Likelihood Discretization}

\subsection*{The transport equation\nopunct} describes the advection of a scalar quantity $\rho_t$ under a velocity field $\vct{u}_t$ as 
\begin{equation}
     \label{eqn:transport}
     \begin{cases}
          \partial_t \rho_t + \divergence (\rho_t \vct{u}_t) = 0 \quad &\text{in $\pdedomain$}, \\
          \rho_t = \rho_t^{\mathrm{in}} \quad & \text{in $\inbound_t$}, \\
          \rho_0 = \rho_{\mathrm{init}} \quad & \text{in $\pdedomain$}.
     \end{cases}
\end{equation}
Here, $\inbound_t \coloneqq \{\vct{x} \in \partial \pdedomain: \vct{u}_t(\vct{x}) \cdot \vct{\nu}_{\pdedomain}(\vct{x}) < 0\}$ is the inflow boundary at time $t$, defined in terms of the outwards pointing normal $\vct{\nu}_{\pdedomain}$ of the domain $\pdedomain \subset \R^d$.
As part of the Euler and Navier-Stokes systems, \cref{eqn:transport} determines the mass density.
In these and many others instances, $\rho_t$ is nonnegative and integrable -- a probability density.
Since $\rho_t$ is the pushforward of $\rho_{\mathrm{init}}$ under the flow map generated by $\vct{u}_t$, it is nonnegative at all times.

\subsection*{The weak form} Integrating against smooth test functions $\te$ yields, for $\outbound \coloneqq \partial \pdedomain \setminus \inbound$, 
\begin{equation}
     \label{eqn:transport_weak}
     \begin{split}
     \partial_t \int \limits_{\pdedomain} \te \rho_t \D \vct{x} 
     - \int \limits_{\pdedomain} \left(\cnst{D}\te\right) \rho_t \vct{u} \D \vct{x} 
     + \int \limits_{\outbound_t} \te \overline{\rho}_t \vct{u}_t \cdot \vct{\nu}_{\pdedomain} \D \vct{x}
     + \int \limits_{\inbound_t} \te \rho^{\mathrm{in}}_t \vct{u}_t \cdot \vct{\nu}_{\pdedomain} \D \vct{x} 
     &= 0 \\
     \iff \ 
     \partial_t \int \limits_{\pdedomain} \te \rho_t \D \vct{x} 
     + \int \limits_{\pdedomain} \te \divergence\left(\rho_t \vct{u}\right) \D \vct{x} 
     - \int \limits_{\inbound_t} \te \rho_t \vct{u}_t \cdot \vct{\nu} \D \vct{x}
     + \int \limits_{\inbound_t} \te \rho^{\mathrm{in}}_t \vct{u}_t \cdot \vct{\nu} \D \vct{x} 
     &= 0.
     \end{split}
\end{equation}

\subsection*{Galerkin methods\nopunct} approximate the solution $\rho$ with an ansatz $\hat{\rho}(\vct{x}, t) = \sum_{i=1}^{m} \trc_i(t) \tr_i(\vct{x})$ in the span of a finite-dimensional basis $\{\tr\}_{1 \leq i \leq m}$ 
and restrict the test functions in \cref{eqn:transport_weak} to the same space, yielding
\begin{equation}
     \label{eqn:galerkin_transport}
     \begin{split}
     \sum_{j = 1}^{m} \dot{\trc}_j \int \limits_{\pdedomain} \phi_j \phi_i \D \vct{x} 
     + \trc_j \int \limits_{\pdedomain} \phi_j \divergence\left(\phi_i \vct{u}_t \right) \D \vct{x} 
     - r_j\int \limits_{\inbound} \tr_i \tr_j \vct{u}_t \cdot \vct{\nu} \D \vct{x}
     &= - \int \limits_{\inbound} \tr_i \rho^{\mathrm{in}}_t \vct{u}_t \cdot \vct{\nu} \D \vct{x}
     , \ \forall 1 \leq i \leq m \ \\
     \iff \ \int \limits_{\pdedomain} \tef \dot{\trf}_t \D \vct{x} 
     + \int \limits_{\pdedomain} \tef \divergence\left(\trf_t \vct{u}_t\right) \D \vct{x} 
     - \int \limits_{\inbound_t} \tef \trf_t \vct{u}_t \cdot \vct{\nu} \D \vct{x}
     &= - \int \limits_{\inbound_t} \tef \rho^{\mathrm{in}}_t \vct{u}_t \cdot \vct{\nu} \D \vct{x},
      \ \forall \tef \in \trs.
     \end{split}
\end{equation}
This \emph{semidiscretization} is amenable to standard time integrators such as Runge-Kutta methods \cite{isherwood2018strong}.

\subsection*{Loss of positivity}
Although the continuous equation \cref{eqn:transport} governing $\rho$ is nonnegativity preserving, this is not generally true for its Galerkin semidiscretization \cref{eqn:galerkin_transport}.
Under large density variations, it can produce negative densities that do not yield physically meaningful pressures, reaction rates, etc.

\subsection*{A statistical perspective} Interpreting the mass density $\rho$ as a probability distribution turns its transport into a statistical inference problem.
Let $\rho_t$ be contained in the parametric model of densities in $\trs$ at time $t$.
A particle sampled from $\rho_t$ in position $\vct{x}$ at time $t$ will find itself at approximately $\vct{x} + \delta_t \vct{u}_t(\vct{x})$ at time $t + \delta_t$.
Determining $\trc(t + \delta_t)$ amounts to choosing a statistical model $P_{\vct{\trc}(t + \delta_t)}$ with density $\sum_{i = 1}^m \trc_i(t + \delta_t) \tr_i$ that best describes the ``data'' obtained from advecting particles distributed according to $\rho_t$.

\begin{remark}
In general, it would be more correct to replace $\vct{x} + \delta_t \vct{u}(\vct{x}, t)$ with $\diff^{\vct{u}}_{\delta_t}(\vct{x})$, where $\diff^{\vct{u}}_{\delta_t}$ is the flow map generated by the vector field $\vct{u}_{(\cdot)}$ for a time interval $\delta_t$.
But since we are interested in small $\delta_t$, we can approximate $\diff^{\vct{u}}_{\delta_t}(\vct{x})$ by $\vct{x} + \delta_t \vct{u}_t(\vct{x})$ without changing our results.
We do so to simplify the notation.
\end{remark}

\subsection*{Moment matching and Galerkin methods} A popular approach for fitting models to data is moment matching, where for a set of test functions $\{\te\}_{1 \leq i \leq n}$, one chooses $\vct{\trc}(t + \delta_t)$ such that the expectations 
$\int \psi_i \D P_{\vct{\trc}}$ match those of the data $\int \psi_i \D P_{\mathrm{data}}$, for all $1 \leq i \leq n$.
For equation~\cref{eqn:transport} with $\vct{u}_t\equiv \zerovct$ on $\R^d \setminus \pdedomain$, this yields
\begin{equation}
     \sum_{i = 1}^{m} r_i(t + \delta_t) \int \limits_{\pdedomain} \phi_i(\vct{x}) \psi(\vct{x}) \D \vct{x} - r_i(t) \int \limits_{\pdedomain} \phi_i(\vct{x}) \psi(\vct{x} + \delta_t \vct{u}) \D \vct{x} 
     = 0, 
     \quad \forall 1 \leq j \leq n.
\end{equation}
Setting $m = n, \te_i = \tr_i$, taking derivatives in $\delta_t=0$ and integrating by parts yields the Galerkin method, whereas a general choice of $n, \te_i$ yields the more general Petrov-Galerkin method.
A well-known shortcoming of the method of moments is that it is not invariant under reparametrization and can return parameter estimates which are unable to produce the observed data.
The Galerkin method inherits these flaws, causing it to produce negative probabilities in the presence of large density-variations.

\subsection*{Maximum likelihood estimation\nopunct} is the dominant paradigm for frequentist parameter estimation.
Different from the method of moments, it is parametrization-invariant and ensures that the observed data is compatible with the inferred model.
It proposes choosing $\vct{\trc}\left(t + \delta_t\right)$ such as to maximize the (log)likelihood of the observed data under $P_{\vct{\trc}(t + \delta_t)}$.
In \cref{eqn:transport} with  $\vct{u}_t\equiv \zerovct$ on $\R^d \setminus \pdedomain$ and writing $P_{\vct{r}(t), \delta_t} \coloneqq \left(\vct{x} \mapsto \vct{x} + \delta_t \vct{u}_t \left(\vct{x}\right)\right)_{\#}$ for the distribution of the advected particles ($(\cdot)_\#$ denotes the pushforward measure), this amounts to solving
\begin{equation}
     \label{eqn:mle_transport}
     \vct{\trc}(t + \delta_t) = \operatorname{arg} \max\limits_{\hat{\vct{\trc}} \in \R^m} 
     \begin{cases}
          \int \limits_{\vct{x} \in \pdedomain} \log\left(\frac{\D P_{\hat{\vct{\trc}}}}{\D P_{\vct{r}(t), \delta_t \vct{u}_t}} \left(\vct{x}\right)\right) \D P_{\vct{r}(t), \delta_t \vct{u}_t}, 
          & \text{if $P_{\vct{\trc}\left(t\right), \delta_t \vct{u}_t} \ll P_{\hat{\vct{\trc}}}$,} \\ 
          - \infty, & \text{else.}
     \end{cases}
\end{equation}
This work proposes practical \emph{maximum likelihood discretizations} derived from this principle and thus \cref{eqn:mle_transport}.

\subsection*{Maximum likelihood discretization and information geometry}
The maximum likelihood estimator \cref{eqn:mle_transport} is equivalently characterized as the projection with respect to the Kullback-Leibler divergence of the observed data onto the statistical model \cite{amari2016information} resulting in
\begin{equation}
     \label{eqn:KL_projection}
          \vct{r}(t + \delta_t) = \operatorname{arg} \min\limits_{\hat{\vct{\trc}} \in \R^m} \left( 
          \DivKL{P_{\vct{\trc}\left(t\right), \delta_t \vct{u}_t}}{P_{\hat{\vct{\trc}}}}
          \coloneqq
          \begin{cases}
          \int \limits_{\vct{x} \in \pdedomain} \log\left(\frac{\D P_{\hat{\vct{\trc}}}}{\D P_{\vct{\trc}\left(t\right), \delta_t \vct{u}_t}} \left(\vct{x}\right)\right) \D P_{\vct{\trc}\left(t\right), \delta_t \vct{u}_t}& \text{if $P_{\vct{\trc}\left(t\right), \delta_t \vct{u}_t} \ll P_{\hat{\vct{\trc}}}$,} \\ 
          - \infty, & \text{else.} \\
          \end{cases}
          \right).
\end{equation}
Compared to the Galerkin method, this method replaces the squared $L^2$ metric with the Kullback-Leibler divergence when determining the element of the model ``closest'' to the data.
The $L^2$ projection onto a manifold $\mathcal{M}$ sends a point $P$ to $\mathrm{Proj}_{L^2}(P)$ along a straight line orthogonal to $\mathcal{M}$ in $\mathrm{Proj}_{L^2}$ with respect to the $L^2$ inner product.
Likewise, the Kullback-Leibler projection onto $\mathcal{M}$ sends a point $P$ to $\operatorname{Proj}_{\mathrm{KL}} (P)$ along a straight line in log-space, a ``dual geodesic'' orthogonal to the manifold with respect to the Fisher-Rao inner product in $\operatorname{Proj}_{\mathrm{KL}} (P)$.
Denoting as $\D P$ the Radon-Nikodym derivative (meaning, the density) of $P$ with respect to the Lebesgue measure, the Fisher-Rao inner product is defined as
\begin{equation}
     \left \langle p, q \right \rangle_{\operatorname{Proj}_{\mathrm{KL}}(P)} \coloneqq \int \limits_{\pdedomain} \frac{p q}{\frac{\D \left(\operatorname{Proj}_{\mathrm{KL}}(P)\right)}{\D \vct{x}}} \D \vct{x} 
\end{equation}
We refer to \cref{fig:galerkin_simplex} for an illustration and to \cite[Chapter 7]{amari2016information} for a detailed discussion.

\section{Continuous Fisher-Rao Galerkin (FRG) semidiscretization}
\label{sec:cfrg}
\subsection*{A maximum likelihood semidiscretization} We now derive a differential equation for $\vct{\trc}$ by taking the derivative with respect to $\delta_t$ in \cref{eqn:mle_transport}.
Integration of this \emph{semidiscretization} in time then yields an approximate solution of \cref{eqn:transport}.
We assume that $\{\tr_i\}_{1 \leq i \leq m}, \{\vct{u}_t\}_{t \geq 0}, \{\rho^{\mathrm{in}}_t\}_{t \geq 0}$ are Lipschitz continuous on $\overline{\pdedomain}$ and that there exists a linear combination of the $\tr_i$ that is lower bounded by a positive number from below.
We denote as $\overline{\rho}_{t}^{\mathrm{in}}, \overline{\vct{u}}_t$ the Lipschitz continuous extensions to $\R^d \setminus \pdedomain$ \cite{kirszbraun1934zusammenziehende}.

\subsection*{Maximality condition} 
We can simplify the maximum likelihood estimator in \cref{eqn:mle_transport} by plugging in the densities of $P_{\hat{\vct{\trc}}}$ and $P_{\vct{\trc}\left(t\right)}$ with respect to the Lebesgue measure.
Since the additive term due to $\int_{\vct{x} \in \pdedomain} \log\left(\D P_{\trc(t)}\left(\vct{x}\right)\right) \D P_{\trc(t)} \vct{x}$ does not depend on $\hat{\vct{r}}$, we obtain $\vct{r}(t + \delta_t)$ as the  minimizer of 
\begin{equation}
     \label{eqn:simple_mle_transport}
      \begin{split}
      \operatorname{arg} \max\limits_{\hat{\vct{r}} \in \R^m} \min_{\lambda \in \R} 
      &\int \limits_{\pdedomain^=} \log\left(\sum_{i = 1}^m \hat{r}_i \phi_i\left(\vct{x} + \delta_t \vct{u}_t\left(\vct{x}\right) \right)\right) \left(\sum_{i = 1}^m r_i(t) \phi_i\left(\vct{x}\right)\right)\D \vct{x} \\
      +&\int \limits_{\pdedomain^+} \log\left(\sum_{i = 1}^m \hat{r}_i \phi_i\left(\vct{x} + \delta_t \overline{\vct{u}}_t\left(\vct{x}\right) \right)\right) \overline{\rho}_{t}^{\mathrm{in}}\left(\vct{x}\right)\D \vct{x} \\
      +& \lambda \left(\int \limits_{\pdedomain} \sum \limits_{i = 1}^{m} \left(\trc_i - \hat{\trc}_i\right) \tr_i\left(\vct{x}\right) \D \vct{x} - \int \limits_{\pdedomain^+} \overline{\rho}_{t}^{\mathrm{in}}(\vct{x}) \D \vct{x} + \int \limits_{\pdedomain-} \left(\sum_{i = 1}^m r_i(t) \phi_i\left(\vct{x}\right)\right) \right).
      \end{split}
\end{equation}
Here, $\lambda$ is a Lagrange multiplier enforcing the global mass conservation, and
\begin{equation*}
      \pdedomain^= \coloneqq \{\vct{x} \in \pdedomain: \vct{x} + \delta_t \vct{u}_t(\vct{x}) \in \pdedomain\}, 
      \ \ \pdedomain^- \coloneqq \{\vct{x} \in \pdedomain: \vct{x} + \delta_t \vct{u}_t(\vct{x}) \notin \pdedomain\}, 
      \ \ \pdedomain^+ \coloneqq \{\vct{x} \in \R^d \setminus \pdedomain: \vct{x} + \delta_t \overline{\vct{u}}_t(\vct{x}) \in \pdedomain\},
\end{equation*}
are the sets of positions of particles that remain in $\pdedomain$, leave $\pdedomain$, and enter $\pdedomain$ under the displacement $\delta_t \vct{u}$, respectively.
Since the objective is concave and the constraint is linear, solutions of this problem are saddle points of the Lagrangian, with the Lagrange multiplier $\lambda$ encoding the global mass conservation.
Taking derivatives with respect to $\hat{\trc}_{k}$ for $1 \leq k \leq m$, we obtain the optimality condition
\begin{equation}
      \label{eqn:optimality_condition_r}
      \int \limits_{\pdedomain^=} \frac{\phi_k\left(\vct{x} + \delta_t \vct{u}_t\left(\vct{x}\right) \right) \left(\sum_{i = 1}^m r_i(t) \phi_i\left(\vct{x}\right)\right)}{\left(\sum_{i = 1}^m \hat{\trc}_i \phi_i\left(\vct{x} + \delta_t \vct{u}_t\left(\vct{x}\right) \right)\right)}\D \vct{x} 
      + \int \limits_{\pdedomain^+} \frac{\phi_k\left(\vct{x} + \delta_t \vct{u}_t\left(\vct{x}\right) \right) \overline{\rho}_{t}^{\mathrm{in}}}{\left(\sum_{i = 1}^m \hat{\trc}_i \phi_i\left(\vct{x} + \delta_t \vct{u}_t\left(\vct{x}\right) \right)\right)}\D \vct{x} 
      - \lambda \int \limits_{\pdedomain} \tr_k(\vct{x}) \D \vct{x} = 0,
\end{equation}
while the derivative with respect to $\lambda$ yields 
\begin{equation}
      \label{eqn:optimality_condition_lambda}
      \int \limits_{\pdedomain} \sum \limits_{i = 1}^{m} \left(\trc_i - \hat{\trc}_i\right) \tr_i\left(\vct{x}\right) \D \vct{x} - \int \limits_{\pdedomain^+} \overline{\rho}_{t}^{\mathrm{in}}(\vct{x}) \D \vct{x} + \int \limits_{\pdedomain-} \left(\sum_{i = 1}^m r_i(t) \phi_i\left(\vct{x}\right)\right) = 0.
\end{equation}

\subsection*{The time-continuous limit}
It is straightforward to see that for $\delta_t = 0$, the solution is given by $\hat{\vct{\trc}} = \vct{\trc}$, and $\lambda = 1$.
We now show that this system also has a solution for nonzero but sufficiently small $\delta_t$ and derive an equation for $\frac{\D \vct{\trc}}{\D \delta_t}$.
Applying the implicit function theorem needs the Jacobian in $\hat{\vct{\trc}} = \vct{\trc}, \lambda = 1, \delta_t = 0$.
The derivative of the left side of \cref{eqn:optimality_condition_r} in $\hat{\trc} = \trc, \lambda = 1, \delta_t = 0$ with respect to $\hat{\trc}_l$ is given by
\begin{equation}
      - \int \limits_{\pdedomain} \frac{\phi_k\left(\vct{x}\right) \phi_l\left(\vct{x}\right)}{\left(\sum_{i = 1}^m \trc_i \phi_i\left(\vct{x}\right)\right)}\D \vct{x}, \quad \text{and with respect to $\lambda$ by} \quad -\int \limits_{\pdedomain} \tr_k(\vct{x}) \D \vct{x}.
\end{equation}
The derivative of the left-hand side of equation \cref{eqn:optimality_condition_lambda} with respect to $\hat{\trc}_l$ is given by $\int_{\pdedomain} \tr_l(\vct{x}) \D \vct{x}$.
The Jacobian of the system is thus given by 
\begin{equation}
      \mtx{J} = 
      - 
      \begin{pmatrix}
       \left(\int \limits_{\pdedomain} \frac{\phi_k\left(\vct{x}\right) \phi_l\left(\vct{x}\right)}{\left(\sum_{i = 1}^m \trc_i \phi_i\left(\vct{x}\right)\right)}\D \vct{x}   \right)_{1 \leq k, l \leq m}   &     \left(\int \limits_{\pdedomain} \tr_k(\vct{x}) \D \vct{x}\right)_{1 \leq k \leq m} \\
       \left(\int \limits_{\pdedomain} \tr_k(\vct{x}) \D \vct{x}\right)_{1 \leq k \leq m}    & 0
      \end{pmatrix}.
\end{equation}
Setting $\mtx{A} = \left(- \int_{\pdedomain} \frac{\phi_k\left(\vct{x}\right) \phi_l\left(\vct{x}\right)}{\left(\sum_{i = 1}^m \trc_i \phi_i\left(\vct{x}\right)\right)}\D \vct{x}   \right)_{1 \leq k, l \leq m}$, 
$\mtx{B} = \left(-\int_{\pdedomain} \tr_k(\vct{x}) \D \vct{x}\right)_{1 \leq k \leq m}$, we can write, if $\mtx{A}$ is invertible, 
\begin{equation}
\mtx{J} =    
\begin{pmatrix}
      \mtx{A} & \mtx{B} \\
      \mtx{B}^T & \mtx{A} 
\end{pmatrix}
=
\begin{pmatrix}
      \Id & \zeromtx \\
      \mtx{B}^\tp \mtx{A}^{-1}  & \Id 
\end{pmatrix}
\begin{pmatrix}
      \mtx{A} & \zeromtx \\
      \zeromtx & -\mtx{B}^\tp \mtx{A}^{-1} \mtx{B} 
\end{pmatrix}
\begin{pmatrix}
      \mtx{A} & \mtx{A}^{-1} \mtx{B} \\
      \zeromtx & \Id
\end{pmatrix}
\end{equation}
If $\sum_{i=1}^{m} \trc_i \tr_i$ is uniformly bounded from above and away from zero, the eigenvalues of $\mtx{A}$ are upper and lower bounded in terms of the $L^2$ Gram matrix of the system $\left(\tr_i\right)_{1 \leq i \leq m}$.
If this system is linearly independent, the invertibility of $\mtx{J}$ follows.
We now define 
\begin{equation}
      \vct{b} \! \coloneqq 
      \! \! \frac{\D }{\D \delta_t} \left.
      \begin{pmatrix}
            \!\! \left(\int \limits_{\pdedomain^=} \frac{\phi_k\left(\vct{x} + \delta_t \vct{u}_t\left(\vct{x}\right) \right) \left(\sum_{i = 1}^m r_i(t) \phi_i\left(\vct{x}\right)\right)}{\left(\sum_{i = 1}^m \trc_i \phi_i\left(\vct{x} + \delta_t \vct{u}_t\left(\vct{x}\right) \right)\right)}\D \vct{x} \!
            +\! \int \limits_{\pdedomain^+} \frac{\phi_k\left(\vct{x} + \delta_t \vct{u}_t\left(\vct{x}\right) \right) \overline{\rho}_{t}^{\mathrm{in}}}{\left(\sum_{i = 1}^m \trc_i \phi_i\left(\vct{x} + \delta_t \vct{u}_{t}\left(\vct{x}\right) \right)\right)}\D \vct{x}
                  \!-\! \lambda \int \limits_{\pdedomain} \tr_k(\vct{x}) \D \vct{x},
            \right)_{1 \leq k \leq m}  \!\!\!\\
      \int \limits_{\pdedomain} \sum \limits_{i = 1}^{m} \left(\trc_i - \hat{\trc}_i\right) \tr_i\left(\vct{x}\right) \D \vct{x} - \int \limits_{\pdedomain^+} \overline{\rho}_{t}^{\mathrm{in}}(\vct{x}) \D \vct{x} + \int \limits_{\pdedomain-} \left(\sum_{i = 1}^m r_i(t) \phi_i\left(\vct{x}\right)\right) 
      \end{pmatrix}
      \!
      \right|_{\delta_t = 0}\!.%\\
\end{equation} 
We compute $\vct{b}$ in multiple steps, beginning with 
\begin{align*}
      \begin{split}
      &\frac{\D}{\D \delta_t}\left.\int \limits_{\pdedomain^=} \frac{\phi_k\left(\vct{x} + \delta_t \vct{u}_t\left(\vct{x}\right) \right) \left(\sum_{i = 1}^m r_i(t) \phi_i\left(\vct{x}\right)\right)}{\left(\sum_{i = 1}^m \trc_i \phi_i\left(\vct{x} + \delta_t \vct{u}_t\left(\vct{x}\right) \right)\right)}\D \vct{x}\right|_{\delta_t = 0}
      = \int \limits_{\pdedomain} \frac{\mathrm{D} \phi_k\left(\vct{x}\right) \vct{u}_t\left(\vct{x}\right)  \left(\sum_{i = 1}^m r_i(t) \phi_i\left(\vct{x}\right)\right)}{\sum_{i = 1}^m \trc_i \phi_i\left(\vct{x}\right)}\D \vct{x}\\
      -& \int \limits_{\pdedomain} \frac{\phi_k\left(\vct{x}\right) \left(\sum_{i = 1}^m r_i(t) \phi_i\left(\vct{x}\right)\right)\left(\sum_{i = 1}^m \trc_i \mathrm{D} \phi_i\left(\vct{x}\right) \vct{u}_t\left(\vct{x}\right) \right)}{\left(\sum_{i = 1}^m \trc_i \phi_i\left(\vct{x}\right)\right)^2}\D \vct{x}  
      - \int \limits_{\outbound_t} \frac{\phi_k\left(\vct{x}\right) \left(\sum_{i = 1}^m r_i(t) \phi_i\left(\vct{x}\right)\right) \vct{u}_t(\vct{x}) \cdot \vct{\nu}(\vct{x})}{\left(\sum_{i = 1}^m \trc_i \phi_i\left(\vct{x}\right)\right)}\D \vct{x}
      \end{split}\\
      =&\phantom{-} \int \limits_{\pdedomain} \mathrm{D} \left(\frac{\phi_k\left(\vct{x}\right) \vct{u}_t\left(\vct{x}\right)}{\sum_{i = 1}^m \trc_i \phi_i\left(\vct{x}\right)}\right) \left(\sum_{i = 1}^m r_i(t) \phi_i\left(\vct{x}\right)\right)\D \vct{x}
      - \int \limits_{\outbound_t} \frac{\phi_k\left(\vct{x}\right) \left(\sum_{i = 1}^m r_i(t) \phi_i\left(\vct{x}\right)\right) \vct{u}_t(\vct{x}) \cdot \vct{\nu}(\vct{x})}{\left(\sum_{i = 1}^m \trc_i \phi_i\left(\vct{x}\right)\right)}\D \vct{x}\\
      =& - \int \limits_{\pdedomain} \frac{\phi_k\left(\vct{x}\right) \divergence \left( \vct{u}_t\left(\vct{x}\right)  \left(\sum_{i = 1}^m r_i(t) \phi_i\left(\vct{x}\right)\right)\right)}{\sum_{i = 1}^m \trc_i \phi_i\left(\vct{x}\right)}\D \vct{x}
      + \int \limits_{\inbound_t} \frac{\phi_k\left(\vct{x}\right) \left(\sum_{i = 1}^m r_i(t) \phi_i\left(\vct{x}\right)\right) \vct{u}_t(\vct{x}) \cdot \vct{\nu}(\vct{x})}{\left(\sum_{i = 1}^m \trc_i \phi_i\left(\vct{x}\right)\right)}\D \vct{x},
\end{align*}
where the last equality follows from integration by parts.Similarly, we compute 
\begin{align}
      \frac{\D}{\D \delta_t}\left. \int \limits_{\pdedomain^+} \frac{\phi_k\left(\vct{x} + \delta_t \overline{\vct{u}}_t\left(\vct{x}\right) \right) \overline{\rho}_{t}^{\mathrm{in}}}{\left(\sum_{i = 1}^m \trc_i \phi_i\left(\vct{x} + \delta_t \vct{u}_t\left(\vct{x}\right) \right)\right)}\D \vct{x}\right|_{\delta_t = 0} 
      = - \int \limits_{\inbound_t} \frac{\phi_k\left(\vct{x}\right) \rho_{t}^{\mathrm{in}} \vct{u}(\vct{x}) \cdot \vct{\nu}(\vct{x})}{\left(\sum_{i = 1}^m \trc_i \phi_i\left(\vct{x}\right)\right)}\D \vct{x}
\end{align}
and 
\begin{align*}
&\frac{\D}{\D \delta_t}\left. - \int \limits_{\pdedomain^+} \overline{\rho}_{t}^{\mathrm{in}}(\vct{x}) \D \vct{x} + \int \limits_{\pdedomain-} \left(\sum_{i = 1}^m r_i(t) \phi_i\left(\vct{x}\right)\right) \D \vct{x} \right|_{\delta_t = 0}  \\
=& \int \limits_{\inbound_t} \rho_{t}^{\mathrm{in}}(\vct{x}) \vct{u}_t\left(\vct{x}\right) \cdot \vct{\nu}\left(\vct{x}\right) \D \vct{x} 
+ \int \limits_{\outbound_t} \left(\sum_{i = 1}^m r_i(t) \phi_i\left(\vct{x}\right)\right) \vct{u}_t\left(\vct{x}\right) \cdot \vct{\nu}\left(\vct{x}\right) \D \vct{x}, 
\end{align*}
resulting in 
\begin{equation}
      \vct{b} = 
      \begin{pmatrix}
            \left(- \int \limits_{\pdedomain} \frac{\phi_k\left(\vct{x}\right) \divergence \left( \vct{u}_t\left(\vct{x}\right)  \left(\sum_{i = 1}^m r_i(t) \phi_i\left(\vct{x}\right)\right)\right)}{\sum_{i = 1}^m \trc_i \phi_i\left(\vct{x}\right)}\D \vct{x}
      - \int \limits_{\inbound_t} \frac{\phi_k\left(\vct{x}\right) \left( \rho_{t}^{\mathrm{in}}\left(\vct{x}\right) - \sum_{i = 1}^m r_i(t) \phi_i\left(\vct{x}\right)\right) \vct{u}_t(\vct{x}) \cdot \vct{\nu}(\vct{x})}{\left(\sum_{i = 1}^m \trc_i \phi_i\left(\vct{x}\right)\right)}\D \vct{x}
      \right)_{1 \leq k \leq m} \\
      \int \limits_{\inbound_t} \overline{\rho}_{t}^{\mathrm{in}}(\vct{x}) \vct{u}_t\left(\vct{x}\right) \cdot \vct{\nu}\left(\vct{x}\right) \D \vct{x} 
      + \int \limits_{\outbound_t} \left(\sum_{i = 1}^m r_i(t) \phi_i\left(\vct{x}\right)\right) \vct{u}_t\left(\vct{x}\right) \cdot \vct{\nu}\left(\vct{x}\right) \D \vct{x}, 
      \end{pmatrix}
\end{equation}
The implicit function theorem guarantees the existence of a unique solution for $\delta_t$ sufficiently small, satisfying
\begin{equation}
      \frac{\D }{\D \delta_t} \left.
      \begin{pmatrix}
            \hat{\vct{\trc}} \\
            \lambda
      \end{pmatrix}
      \right|_{\delta_t = 0} = \mtx{J}^{-1} \vct{b}.
\end{equation}
Guessing $\frac{\D}{\D \delta_t} \lambda = 0$, this implies that $\dot{\vct{\trc}} \coloneqq \frac{\D }{\D \delta_t} \left. \hat{\vct{\trc}} \right|_{\delta_t = 0}$ is the unique solution of
\begin{equation}
      \label{eqn:first_condition}
      \begin{split}
            \forall 1 \leq j \leq m: 
            \int \limits_{\pdedomain} \frac{\phi_j\left(\vct{x} \right) \left(\sum_{i = 1}^m \dot{\trc}_i \phi_i\left(\vct{x}\right)\right)}{\left(\sum_{i = 1}^m \trc_i \phi_i\left(\vct{x} \right)\right)}\D \vct{x} 
            + &\int \limits_{\pdedomain} \frac{\phi_j\left(\vct{x}\right) \divergence \left( \vct{u}_t\left(\vct{x}\right)  \left(\sum_{i = 1}^m r_i(t) \phi_i\left(\vct{x}\right)\right)\right)}{\sum_{i = 1}^m \trc_i \phi_i\left(\vct{x}\right)}\D \vct{x} \\
            + &\int \limits_{\inbound_t} \frac{\phi_j\left(\vct{x}\right) \left( \rho_{t}^{\mathrm{in}}\left(\vct{x}\right) - \sum_{i = 1}^m r_i(t) \phi_i\left(\vct{x}\right)\right) \vct{u}_t(\vct{x}) \cdot \vct{\nu}(\vct{x})}{\left(\sum_{i = 1}^m \trc_i \phi_i\left(\vct{x}\right)\right)}\D \vct{x} 
            = 0.
      \end{split}
\end{equation}
To verify that this guess is correct, we have to ensure that 
\begin{equation}
      \int \limits_{\pdedomain} \sum \limits_{i = 1}^{m} \dot{\trc}_i \tr_i\left(\vct{x}\right) \D \vct{x} 
      +\int \limits_{\inbound_t} \overline{\rho}_{t}^{\mathrm{in}}(\vct{x}) \vct{u}_t\left(\vct{x}\right) \cdot \vct{\nu}\left(\vct{x}\right) \D \vct{x} 
      + \int \limits_{\outbound_t} \left(\sum_{i = 1}^m r_i(t) \phi_i\left(\vct{x}\right)\right) \vct{u}_t\left(\vct{x}\right) \cdot \vct{\nu}\left(\vct{x}\right) \D \vct{x} = 0. 
\end{equation}
This follows by linear combination of \cref{eqn:first_condition} with weights $\left(\hat{\trc}_j\right)_{1 \leq j \leq m} = \left(\trc_j\right)_{1 \leq j \leq m}$ and integration by parts. 
\subsection*{The Fisher-Rao Galerkin method} 
Analogous to equation \cref{eqn:galerkin_transport}, we can write the equation for $\vct{\trc}$ as
\begin{equation}
      \begin{split}
      \label{eqn:FRgalerkin_transport}
      \sum_{j = 1}^{m} \dot{\trc}_j \int \limits_{\pdedomain} \frac{\phi_j \phi_i}{\sum_{k = 1}^m \trc_k \phi_k} \D \vct{x} + \trc_j \int \limits_{\pdedomain} \frac{\phi_i \divergence\left(\phi_j \vct{u}_t \right)}{{\sum_{k = 1}^m \trc_k \phi_k}} \D \vct{x} 
      - \trc_j \int \limits_{\inbound_t} \frac{\phi_j \phi_i \vct{u}_t \cdot \vct{\nu}}{\sum_{k = 1}^m \trc_k \phi_k}\D \vct{x} 
            &= -\int \limits_{\inbound_t} \frac{\phi_i \rho_{t}^{\mathrm{in}}\vct{u}_t \cdot \vct{\nu}}{\sum_{k = 1}^m \trc_k \phi_k}\D \vct{x}, \forall 1 \leq i \leq m \\ 
      \iff \quad \int \limits_{\pdedomain} \frac{\tef \dot{\trf}_t}{\trf_t} \D \vct{x} 
            + \int \limits_{\pdedomain} \frac{\tef \divergence\left(\trf_t \vct{u}_t\right)}{\trf_t}\D \vct{x} 
            - \int \limits_{\inbound_t} \frac{\tef \trf_t \vct{u}_t \cdot \vct{\nu}}{\trf_t} 
            &= - \int \limits_{\inbound_t} \frac{\tef \rho_{t}^{\mathrm{in}} \vct{u}_t \cdot \vct{\nu}}{\trf_t}, \ \forall \tef \in \trs .
      \end{split}
\end{equation}
This replaces the $L^2$ with the Fisher-Rao inner product, emphasizing \emph{relative}, instead of absolute errors.
\subsection*{A geometric view on Fisher-Rao Galerkin}
Solutions of the transport equation, a hyperbolic PDE, evolves along characteristic lines. 
Their evolution is generally not constrained to any finite dimensional Galerkin space. 
In the 1D example in \cref{fig:galerkin_projection}, we start with an initial condition within the Galerkin space of piecewise constant functions, according to the grid in black.
As the solution evolves, along the characteristic curves (orange), it leaves the Galerkin space.
\cref{fig:galerkin_projection_step} illustrates how the Galerkin method projects the solution back onto the Galerkin space, so only the coefficients $r_i$ need to be tracked.
It uses finite time steps for illustration purposes but in semidiscretizations like \cref{eqn:galerkin_transport}, the projection happens continuously and the solution never leaves the Galerkin space.
The traditional Galerkin method applies an $L^2$ projection at each time step corresponding to the Method of Moments. 
\cref{fig:galerkin_simplex} illustrates that while the time evolution preserves positivity, the projection step does not, in general \cite{nochetto2002positivity}.
Instead, the Fisher-Rao Galerkin method applies a Fisher-Rao projection at each time step. This amounts to continuously performing maximum likelihood estimation to determine the coefficients $r_i$ that best represent the solution.

\begin{figure}[H]

      \centering

      \includegraphics{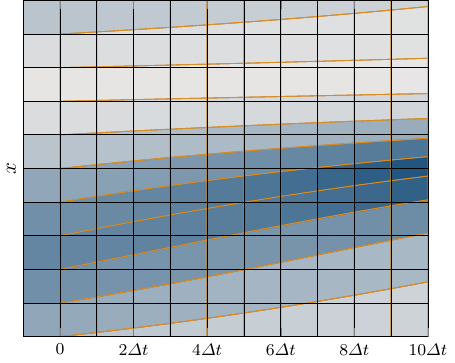}

    \caption{The solution leaves the Galerkin space (shown as mesh), even if the initial condition is contained in it.}
    \label{fig:galerkin_projection}

\end{figure}

\begin{figure}[H]

      \centering

      \includegraphics{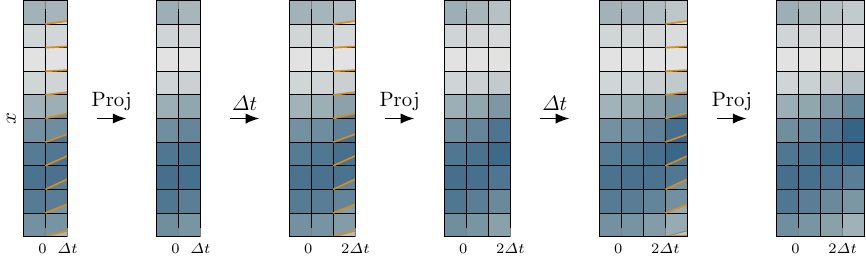}

      \caption{Galerkin projection of the solution to transport equation after each time step $\Delta t$.}

      \label{fig:galerkin_projection_step}

\end{figure}

\begin{figure}[H]

      \centering

      \includegraphics{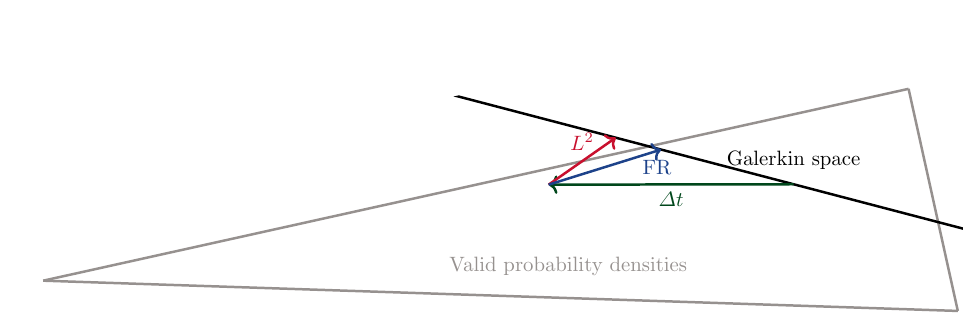}

      \caption{   A 2D simplex representing subset of feasible probability distributions, intersected by 2D subset of Galerkin space, which contains feasible approximate solutions.
      A solution initial in the Galerkin space exits it after a time-step, requiring projection back onto the Galerkin space.
      In this example, the $L^{2}$ projection may leave the simplex (violating probability constraints), while a Fisher-Rao (FR) projection remains in the simplex structure, ensuring a valid probability distribution.}

      \label{fig:galerkin_simplex}

\end{figure}

\subsection*{A nonlinear projection} Negative results by \cite{nochetto2002positivity} show that in the case of finite elements (where $\trs$ consists of piecewise polynomial function), there is no linear projection operator that preserves positivity.
Instead, of linear projections with respect to the $L^2$ inner product, maximum likelihood estimation performs a nonlinear projection with respect to a Bregman divergence --- the Kullback-Leibler divergence \cref{eqn:KL_projection}.
In the time-continuous limit, this projection is equivalent to projection with respect to the Fisher-Rao metric, the second order approximation of $\DivKL{\rho}{\sigma}$ in $\rho \approx \sigma$. 
\cref{sec:kl_approximation} shows that just like projection using inner products yields error bounds in the associated metric, projection with the Fisher-Rao metric yields error bounds in the associated Bregman divergence.

\section{Discontinuous Fisher-Rao Galerkin (DFRG) semidiscretization}
\subsection*{Discontinuous Galerkin (DG) methods\nopunct} overcome two shortcomings of Galerkin methods on transport problems: the lack of local conservation and the requirement to invert a mass matrix when computing $\dot{\vct{\trc}}$.
They achieve this by partitioning the domain $\pdedomain = \bigcup_{k = 1}^{n} \overline{\omega}_{k}$ into disjoint open sets $\left(\omega_{k}\right)_{n}$, such that every basis function $\tr_i$ is supported on a single $\omega_{k}$.
We write $\vct{o} \in \{1 \ldots n\}^{m}$ for the vector assigning the $m$ basis function indices to the $n$ domains such that $\tr_i$ is supported on $\omega_{o_i}$.
Discontinuous Galerkin methods integrate $\sum_{i = 1}^{m} \dot{\trc}_i \tr + \trc_i \divergence\left(\tr_i \vct{u}_t \right) = 0$ against $\tr_j$ only over $\omega_{o_j}$.
Integration by parts then yields
\begin{align}
     \label{eqn:dgalerkin_transport}
     \sum_{i = 1}^{m} \dot{\trc}_i \int \limits_{\omega_{o_i}} \phi_i \phi_j \D \vct{x} 
     - \trc_i \int \limits_{\omega_{o_i}} \cnst{D} \phi_i \phi_j \vct{u}_t \D \vct{x} 
     + \int \limits_{\partial \omega_{o_i}} \tr_i \phi_j \vct{u}_t \cdot \vct{\nu} \D \vct{x} = 0, \ \forall 1 \leq j \leq m \\
      \iff \ \int \limits_{\omega_{k}}\dot{\trf}_t \tef \D \vct{x} - \int \limits_{\omega_{k}} [\cnst{D} \tef] \trf_t \vct{u}_t \D \vct{x} + \int \limits_{\partial \omega_{k}} \tef \trf_t \vct{u}_t \cdot \vct{\nu} \D \vct{x}= 0, \ \forall \tef \in \trs, 1 \leq k \leq n,
\end{align}
for $\vct{\nu}$ the outward pointing unit normal vector on $\partial \omega$.
At this point, the weak formulations on the different $\omega_{k}$ are completely independent.
DG methods relate the problems on different subdomains by replacing the $\trf \vct{u}_t \cdot \vct{\nu}$ in the boundary integrals with a \emph{numerical flux} $f: \partial \omega_k \rightarrow \R$.
To ensure local conservation, they require that along the shared boundary of two domains $\omega^+,\omega^-$ the numerical flux satisfies $f^{\omega_+, \vct{u}, \trf} = -f_{\omega_-, \vct{u}, \trf}$.
Numerous numerical fluxes have been proposed, all of which can be viewed as approximate solution of a Riemann problem in the boundary layer.
Most popular are the Lax-Friedrichs and upwind fluxes,
\begin{equation}
\label{eqn:fluxes}
f_{\mathrm{LF}}^{\omega_k, \trf, \vct{u}_t}(x) = \frac{\trf_+ + \trf_-}{2} \vct{u}_t \cdot \vct{\nu} - \alpha \frac{\trf_+ - \trf_-}{2}, 
\qquad f_{\mathrm{up}}^{\omega_k, \trf, \vct{u}_t}(x) = \trf_{-\operatorname{sign}\left(\vct{u}_t \cdot \vct{\nu}\right)} \vct{u}_t \cdot \vct{\nu}.
\end{equation}
Here, $\alpha \in \R_+$ modulates the numerical dissipation of the Lax-Friedrichs flux.
The $\trf^{\pm}$ are the values of $\trf$ on $\omega_{\pm}$ and the interface-normal $\vct{\nu}$ points from $\omega_{-}$ to $\omega_{+}$.
Treating inflow conditions $\rho_{t}^{\mathrm{in}}$ as boundary values on outward cell interfaces $\partial \omega_{k} \cap \partial \inbound_t$, the resulting discontinuous Galerkin method is
\begin{equation}
     \int \limits_{\omega_{k}}\dot{\trf} \tef \D \vct{x} - \int \limits_{\omega_{k}} [\cnst{D} \tef] \trf \vct{u}_t \D \vct{x} + \int \limits_{\partial \omega_{k}} \tef f^{\omega_k, \trf, \vct{u}_t} \D \vct{x}= 0, \ \forall \tef \in \trs, 1 \leq k \leq n,
\end{equation}
where the numerical flux $f$ depends on the current solution $\trf$ in $\omega_{k}$ and adjacent cells, as well as $\omega_{k}$ and $\vct{u}_t$.

\subsection*{Benefits of DG methods} 
DG methods have two key benefits over continuous Galerkin methods.
The first is that DG methods with $\left\{\indicator_{\omega_{k}}\right\}_{1 \leq k \leq n} \subset \{\tr_{j}\}_{1 \leq j \leq m}$ are locally conservative, meaning that the integral of $\trf$ in each $\omega_{k}$ is constant, up to the integral over the numerical flux on the boundary.
In other words, the total amount of mass in each $\omega_k$ is constant, up to the changes due to flow across the boundary.
This follows from setting $\tr_j = \indicator_{\omega_{k}}$ for $1 \leq k \leq n$ in \eqref{eqn:dgalerkin_transport}.
The second benefit of DG methods is that the mass matrix
\begin{equation}
\mtx{M} \coloneqq 
\left(\int \limits_{\Omega} \phi_i \phi_j \D \vct{x}\right)_{1 \leq i, j \leq m} 
= \left(\int \limits_{\omega_{o_i}} \phi_i \phi_j \D \vct{x}\right)_{1 \leq i, j \leq m}
= \left(\int \limits_{\omega_{o_j}} \phi_i \phi_j \D \vct{x}\right)_{1 \leq i, j \leq m}
\end{equation}
is block diagonal, with each block corresponding to a single $\omega_{k}$.
Since the number of basis functions in each $\omega_{k}$ is a small and constant under refinement, this greatly simplifies the computation of $\dot{\vct{\trc}}$.

\subsection*{Discontinuous Fisher-Rao Galerkin (DFRG) methods} 
Applying the argument of \cref{sec:cfrg} to a single $\omega$, replacing Lipschitz extensions $\overline{\rho}^{\mathrm{in}}_t,\overline{\vct{u}}$ with the values on neighboring cells, we obtain $\forall 1 \leq i \leq m$, 
\begin{equation}
      \begin{split}
      \label{eqn:DFRgalerkin_transport_single_domain}
      \sum_{j = 1}^{m} \dot{\trc}_j \int \limits_{\omega} \frac{\phi_j \phi_i}{\sum_{k = 1}^m \trc_k \phi_k} \D \vct{x} + \trc_j \int \limits_{\omega} \frac{\phi_i \divergence\left(\phi_j \vct{u}_t \right)}{{\sum_{k = 1}^m \trc_k \phi_k}} \D \vct{x} 
      - \trc_j \int \limits_{\widecheck{\gamma}_{t}} \frac{\phi_j \phi_i \vct{u}_t \cdot \vct{\nu}}{\sum_{k = 1}^m \trc_k \phi_k}\D \vct{x} 
      &= -\int \limits_{\widecheck{\gamma}_t} \frac{\phi_i \rho_{t}^{\mathrm{in}}\vct{u}_t \cdot \vct{\nu}}{\sum_{k = 1}^m \trc_k \phi_k}\D \vct{x}, \\
      \iff \quad \int \limits_{\omega} \frac{\tef \dot{\trf}}{\trf} \D \vct{x} 
      + \int \limits_{\omega} \frac{\tef \divergence\left(\trf \vct{u}_t\right)}{\trf}\D \vct{x} 
      - \int \limits_{\widecheck{\gamma}_t} \frac{\tef \trf \vct{u}_t \cdot \vct{\nu}}{\trf} 
      &= -\int \limits_{\widecheck{\gamma}_t} \frac{\tef \trf \vct{u}_t \cdot \vct{\nu}}{\trf}, \ \forall \tef \in \trs,
      \end{split}
\end{equation}
where, as before, $\widecheck{\gamma}_t$ is the inflow boundary of $\omega$.
Integration by parts yields, $\forall 1 \leq i \leq m$,
\begin{equation}
      \begin{split}
      \label{eqn:DFRgalerkin_transport_single_domain_IBP}
      \sum_{j = 1}^{m} 
          \dot{\trc}_j \int \limits_{\omega} \frac{\phi_j \phi_i}{\sum_{k = 1}^m \trc_k \phi_k} \D \vct{x} 
          - \trc_j \int \limits_{\omega} \left[\cnst{D} \frac{\phi_i }{{\sum_{k = 1}^m \trc_k \phi_k}}\right] \phi_j \vct{u} \D \vct{x} 
          + \trc_j \int \limits_{\widehat{\gamma}_{t}} \frac{\phi_j \phi_i \vct{u}_t \cdot \vct{\nu}}{\sum_{k = 1}^m \trc_k \phi_k}\D \vct{x} 
      &= -\int \limits_{\widecheck{\gamma}_t} \frac{\phi_i \rho_{t}^{\mathrm{in}}\vct{u}_t \cdot \vct{\nu}}{\sum_{k = 1}^m \trc_k \phi_k}\D \vct{x}, \\
      \iff \quad \int \limits_{\omega} \frac{\tef \dot{\trf}}{\trf} \D \vct{x} 
          - \int \limits_{\omega} \left[\cnst{D} \frac{\tef}{\trf}\right]\trf \vct{u}\D \vct{x} 
          + \int \limits_{\widehat{\gamma}_t} \frac{\tef \trf \vct{u}_t \cdot \vct{\nu}}{\trf} 
          &= - \int \limits_{\widecheck{\gamma}_t} \frac{\tef \trf \vct{u}_t \cdot \vct{\nu}}{\trf}, \ \forall \tef \in \trs.
      \end{split}
\end{equation}
Here, $\widehat{\gamma}_t \coloneqq \partial \omega \setminus \widecheck{\gamma}_t$.
Summing up the terms obtained for each $\omega$ results in 
\begin{align}
     \label{eqn:DFRgalerkin_transport}
     \quad \int \limits_{\omega_{k}} \frac{\dot{\trf} \tef}{\trf} \D \vct{x} 
     - \int \limits_{\omega_{k}} \left[\cnst{D} \frac{\tef }{\trf}\right]\trf \vct{u}_t \D \vct{x} 
     + \int \limits_{\partial \omega_{k}} \frac{\tef f_{\mathrm{up}}^{\omega_k, \trf, \vct{u}_t}}{\trf}\D \vct{x} 
     = 0, \ \forall \tef \in \trs, 1 \leq k \leq n.
\end{align}
\begin{remark}
     By doing the analog computation for the method of moments with test functions $\te = \tr$, one can likewise derive ordinary discontinuous Galerkin methods with upwind flux.
\end{remark}
\begin{remark}
     If the velocity $\vct{u}_t$ has jumps across cell interfaces, for instance due to being discontinuously Galerkin discretized itself, the same derivation yields (for both ordinary and Fisher-Rao DG), the flux 
     \begin{equation}
          f_{\mathrm{kin}}^{k, \trf, \vct{u}_t}(x) 
          = \trf_{-} \max\left(\vct{u}^{-}_t \cdot \vct{\nu}, 0\right)
          + \trf_{+} \min\left(\vct{u}^{+}_t \cdot \vct{\nu}, 0\right).
     \end{equation}
     It is closely related to the kinetic flux vector splitting for Euler equations due to \cite{pullin1980direct,mandal1994kinetic}.
\end{remark}

\subsection*{Benefits inherited from DG}
Discontinuous Fisher-Rao Galerkin inherits the benefits of both DG and Fisher-Rao Galerkin methods.
First, it is straightforward to see that the Fisher-Rao mass matrix 
\begin{equation}
\mtx{M}^{\trf} \coloneqq 
\left(\int \limits_{\Omega} \frac{\phi_i \phi_j}{\trf} \D \vct{x}\right)_{1 \leq i, j \leq m} 
= \left(\int \limits_{\omega_{o_i}} \frac{\phi_i \phi_j}{\trf} \D \vct{x}\right)_{1 \leq i, j \leq m}
= \left(\int \limits_{\omega_{o_j}} \frac{\phi_i \phi_j}{\trf} \D \vct{x}\right)_{1 \leq i, j \leq m},
\end{equation}
shares the block-diagonal structure of the DG mass matrix $\mtx{M}$.
Furthermore, a linear combination of equations~\cref{eqn:DFRgalerkin_transport} with coefficients $\vct{\trc}$, thus using $\trf$ as a test function, implies that the mass in each $\omega_{k}$ is conserved up to fluxes through the boundary.
Unlike DG, this does not require $\left\{\indicator_{\omega_{k}}\right\}_{1 \leq k \leq n} \subset \{\tr_{j}\}_{1 \leq j \leq m}$.
Conservation of $\omega_k$ follows by choosing the restriction $\tef = \trf\big|_{\omega_k}$ in equation~\cref{eqn:DFRgalerkin_transport} as test function.
By construction, this function is in the test spaces, irrespective of the choice of basis functions.
In the next section, we will showcase the theoretical benefits of (D)FRG methods over conventional (D)G approaches.

\section{KL-approximation properties of (D)FRG semidiscretizations}
\label{sec:kl_approximation}
\subsection*{The Kullback-Leibler (KL) divergence} 
A key measure of quality of a semidiscretization is how quickly solutions $\trf$ of the resulting ODE diverge from the true solution $\rho$ of the PDE.
Ordinary Galerkin methods are defined as projections with respect to an inner product.
Their error estimates are naturally formulated in terms of the norm induced by this inner product.
(Discontinuous) Fisher-Rao Galerkin methods are projections with respect to the Fisher-Rao inner product, which depends on the current solution $\trf$.
Thus, no single norm can adequately capture the resulting error estimate.
Instead, we will capture the error in terms of the (generalized) Kullback-Leibler divergence.
\begin{definition}
   The (generalized) Kullback-Leibler (KL) divergence $\DivKL{\rho}{\sigma}$ between two (unnormalized) probability measures $P$ and $Q$ is defined as 
   \begin{equation}
        \DivKL{P}{Q} \coloneqq 
        \begin{cases}
        \int \log\left(\frac{\D P}{\D Q}\left(\vct{x}\right)\right) \D P \left(\vct{x}\right) - \int \D P(\vct{x}) + \int \D Q(\vct{x}), \quad &\text{if $P$ has a density $\frac{\D P}{\D Q}$ with respect to $Q$,}\\ 
        \! \! \qquad \qquad \qquad \qquad \qquad \infty, &\text{else}.
        \end{cases}
   \end{equation}
   For two densities $\rho$, $\sigma$ with respect to a shared reference measure $R$, we abuse notation to write
   \begin{equation}
        \DivKL{\rho}{\sigma} \coloneqq \DivKL{\rho \D R}{\sigma\D R}
        = \int \log\left(\frac{\rho\left(\vct{x}\right)}{\sigma \left(\vct{x}\right)}\right) \rho\left(\vct{x}\right) \D R \left(\vct{x}\right)
           - \int \rho(\vct{x}) \D R(\vct{x}) 
           + \int \sigma(\vct{x}) \D R(\vct{x}).
   \end{equation}
\end{definition}

A key property of the KL divergence between densities is its invariance under changes of the reference measures, so long as the densities are adjusted to represent the same distribution.
This is a special case of the parametrization invariance of the KL divergence that uniquely characterizes it among all Bregman divergences \cite{amari2016information}.
We now show that the KL divergence interacts well with the transport equation.
\begin{lemma}
    \label{lem:kl_bound_aux}
    Let $\rho_t > 0$ be a weak solution of the transport equation, in the sense of \cref{eqn:transport_weak} and $\sigma_t$ a positive, time-dependent density (that need not satisfy a transport equation.) 
    Let $\sigma_t^{\mathrm{in}} > 0$ be any function on $\inbound_t$, that need not be related to $\sigma$ in any way.
    Then, the KL divergence of $\rho_t$ and $\sigma_t$ satisfies
    \begin{equation}
        \begin{aligned}
            \frac{\D}{\D t} \DivKL{\rho_t}{\sigma_t} =
            &- \int \limits_{\pdedomain} \frac{ \rho_t \left(\dot{\sigma}_t + \divergence\left(\sigma_t \vct{u}_t\right)\right)}{\sigma_t}\D \vct{x} 
            + \int \limits_{\inbound_t} \frac{\rho_t \sigma_t \vct{u}_t \cdot \vct{\nu}}{\sigma_t}\D \vct{x}
            - \int \limits_{\inbound_t} \frac{\rho_t\sigma^{\mathrm{in}}_t \vct{u}_t \cdot \vct{\nu} }{\sigma_t} \D \vct{x}\\
            & + \int \limits_{\pdedomain} \dot{\sigma}_t \D \vct{x} + \int \limits_{\outbound} \sigma_t \vct{u} \cdot \vct{\nu}_t \D \vct{x}  + \int \limits_{\inbound_t} \sigma_t^{\mathrm{in}} \vct{u}_t \cdot \vct{\nu} \D \vct{x} \\
            &+ \int \limits_{\outbound_t} \left(\log\left(\frac{\rho_t}{\sigma_t}\right) \rho_t  - \rho_t + \sigma_t\right) \left(-\vct{u}_t \cdot \vct{\nu}\right) \D \vct{x}\\
            &  + \int \limits_{\inbound_t} \left(\log\left(\frac{\rho_t}{\sigma_t}\right) \rho^{\mathrm{in}}_t           
            + \left(\frac{\rho_t}{\sigma_t} - \frac{\sigma_t}{\sigma_t}\right) \sigma_t^{\mathrm{in}}\right) \left(- \vct{u}_t \cdot \vct{\nu}\right)\D \vct{x}
        \end{aligned}
    \end{equation}
\end{lemma}
In the above, the first two lines of the right-hand side are residuals of plugging $\sigma_t$ into \cref{eqn:transport_weak}, with inflow given by $\sigma_t^{\mathrm{in}}$, for $\te \equiv \rho_t / \sigma_t$ and $\te \equiv 1$ respectively.
The third accounts for the KL divergence of the flow leaving the domain along the outflow boundary and the fourth for the change in KL divergence due to inflow.
\begin{proof}
    We compute the time derivative as 
\begin{align}
    & \frac{\D}{\D t} \DivKL{\rho_t}{\sigma_t} = \frac{\D}{\D t} \left(\int \limits_{\pdedomain} \log\left(\frac{\rho_t}{\sigma_t}\right) \rho_t \D \vct{x} 
    - \int \limits_{\pdedomain} \rho_t \D \vct{x} 
    + \int \limits_{\pdedomain} \sigma_t \D \vct{x}
    \right)\\
    &\phantom{\frac{\D}{\D t} \DivKL{\rho_t}{\sigma_t}}
    = \int \limits_{\pdedomain} \left(\frac{\dot{\rho}_t}{\rho_t} - \frac{\dot{\sigma}_t}{\sigma_t}\right) \rho_t \D \vct{x}
    + \int \limits_{\pdedomain} \log\left(\frac{\rho_t}{\sigma_t}\right)  \dot{\rho}_t \D \vct{x} - \int \limits_{\pdedomain} \dot{\rho}_t \D \vct{x} 
    + \int \limits_{\pdedomain} \dot{\sigma}_t \D \vct{x}\\
    &\phantom{\frac{\D}{\D t} \DivKL{\rho_t}{\sigma_t}}
    = - \int \limits_{\pdedomain}  \frac{\dot{\sigma}_t}{\sigma_t} \rho_t \D \vct{x}
    + \int \limits_{\pdedomain} \log\left(\frac{\rho_t}{\sigma_t}\right)  \dot{\rho}_t \D \vct{x} + \int \limits_{\pdedomain} \dot{\sigma}_t \D \vct{x}\\
    &
    \begin{aligned}
    \phantom{\frac{\D}{\D t} \DivKL{\rho_t}{\sigma_t}}
    = &- \int \limits_{\pdedomain}  \frac{\dot{\sigma}_t}{\sigma_t} \rho_t \D \vct{x}
    + \int \limits_{\pdedomain} \left(\frac{\cnst{D} \rho_t}{\rho_t} - \frac{\cnst{D}\sigma_t}{\sigma_t}\right)  \rho_t \vct{u}_t \D \vct{x} 
    + \int \limits_{\pdedomain} \dot{\sigma}_t \D \vct{x} \\
    &- \int \limits_{\outbound_t} \log\left(\frac{\rho_t}{\sigma_t}\right) \rho_t \vct{u}_t \cdot \vct{\nu} \D \vct{x}
    - \int \limits_{\inbound_t} \log\left(\frac{\rho_t}{\sigma_t}\right) \rho_t^{\mathrm{in}} \vct{u}_t \cdot \vct{\nu} \D \vct{x}
    \end{aligned}\\
    &
    \begin{aligned}
    \phantom{\frac{\D}{\D t} \DivKL{\rho_t}{\sigma_t}}
    = &- \int \limits_{\pdedomain}  \frac{\dot{\sigma}_t}{\sigma_t} \rho_t \D \vct{x}
    + \int \limits_{\pdedomain} \left(\frac{\divergence\left( \rho_t \vct{u}_t\right)}{\rho_t} - \frac{\divergence\left(\sigma_t \vct{u}_t\right)}{\sigma_t}\right)  \rho_t \D \vct{x} 
    + \int \limits_{\pdedomain} \dot{\sigma}_t \D \vct{x} \\
    &- \int \limits_{\outbound_t} \log\left(\frac{\rho_t}{\sigma_t}\right) \rho_t \vct{u}_t \cdot \vct{\nu} \D \vct{x}
    - \int \limits_{\inbound_t} \log\left(\frac{\rho_t}{\sigma_t}\right) \rho_t^{\mathrm{in}} \vct{u}_t \cdot \vct{\nu} \D \vct{x}
    \end{aligned}\\
    &
    \begin{aligned}
    \phantom{\frac{\D}{\D t} \DivKL{\rho_t}{\sigma_t}}
    = &- \int \limits_{\pdedomain}  \frac{\dot{\sigma}_t}{\sigma_t} \rho_t \D \vct{x}
    - \int \limits_{\pdedomain} \frac{\divergence\left(\sigma_t \vct{u}_t\right)}{\sigma_t}  \rho_t \D \vct{x} 
    + \int \limits_{\partial \pdedomain} \rho_t \vct{u}_t \cdot \vct{\nu} \D \vct{x} 
    + \int \limits_{\pdedomain} \dot{\sigma}_t \D \vct{x} \\
    &- \int \limits_{\outbound_t} \log\left(\frac{\rho_t}{\sigma_t}\right) \rho_t \vct{u}_t \cdot \vct{\nu} \D \vct{x}
    - \int \limits_{\inbound_t} \log\left(\frac{\rho_t}{\sigma_t}\right) \rho_t^{\mathrm{in}} \vct{u}_t \cdot \vct{\nu} \D \vct{x}
    \end{aligned}\\
    &
    \begin{aligned}
    \phantom{\frac{\D}{\D t} \DivKL{\rho_t}{\sigma_t}}
    = &- \int \limits_{\pdedomain}  \frac{\rho_t \left( \dot{\sigma}_t + \divergence\left(\sigma_t \vct{u}_t\right)\right)}{\sigma_t} \D \vct{x} + \int \limits_{\inbound_t} \frac{\rho_t \sigma_t \vct{u}_t \cdot \vct{\nu}}{\sigma_t}\D \vct{x}
    - \int \limits_{\inbound_t} \frac{\rho_t\sigma^{\mathrm{in}}_t \vct{u}_t \cdot \vct{\nu} }{\sigma_t} \D \vct{x}\\ 
    &+ \int \limits_{\pdedomain} \dot{\sigma}_t \D \vct{x} + \int \limits_{\outbound} \sigma_t \vct{u} \cdot \vct{\nu}_t \D \vct{x}  + \int \limits_{\inbound_t} \sigma_t^{\mathrm{in}} \vct{u}_t \cdot \vct{\nu} \D \vct{x} \\
    & - \int \limits_{\outbound_t} \log\left(\frac{\rho_t}{\sigma_t}\right) \rho_t \vct{u}_t \cdot \vct{\nu} \D \vct{x}
    + \int \limits_{\outbound_t} \rho_t \vct{u}_t \cdot \vct{\nu} \D \vct{x} 
    - \int \limits_{\outbound} \sigma_t \vct{u} \cdot \vct{\nu}_t \D \vct{x}  \\
    &    - \int \limits_{\inbound_t} \log\left(\frac{\rho_t}{\sigma_t}\right) \rho_t^{\mathrm{in}} \vct{u}_t \cdot \vct{\nu} \D \vct{x}
    + \int \limits_{\inbound_t} \frac{\rho_t\sigma_t^{\mathrm{in}} \vct{u}_t \cdot \vct{\nu} }{\sigma_t} \D \vct{x}- \int \limits_{\inbound_t} \sigma_t^{\mathrm{in}} \vct{u}_t \cdot \vct{\nu} \D \vct{x}.
    \end{aligned}
\end{align}
In the last step above we have added and subtracted the missing boundary terms to complete the residuals in the first and second line of the right-hand side.
By reordering the terms in each line, we obtain the result.
\end{proof}
We will use this lemma to bound the rate of divergence of (approximate) solutions of the transport equation.
We first consider the case where $\sigma_t$ satisfies the same transport equation (although with possibly different initial condition).

\subsection*{KL non-divergence}
A crucial property of the KL-divergence applied to transport equations, is that it is constant under the action of the transport equation, up to in/outflow terms.
This is a well-known result in the theory of continuous Markov processes.
It is also closely related to the so-called H-theorem \cite{morimoto1963markov}.

\begin{corollary}
    Denote as $\rho_t$ and $\sigma_t$ the solutions of the transport equation \cref{eqn:transport_weak} with different initial conditions and the same $\vct{u}, \rho_{t}^{\mathrm{in}}$.
    Then, the KL divergence $\DivKL{\rho_t}{\sigma_t}$ satisfies
    \begin{equation}
    \begin{aligned}
        \frac{\D}{\D t} \DivKL{\rho_t}{\sigma_t} 
        = &\int \limits_{\inbound_t} \frac{\D}{\D s} \left.\log\left(\frac{\rho_t + s}{\sigma_t + s}\right) \left(\rho_t + s\right)\right|_{s = 0} \left(-\rho_t^{\mathrm{in}} \vct{u}_t \cdot \vct{\nu}\right) \D \vct{x}\\
            + &\int \limits_{\outbound_t} \left(\log\left(\frac{\rho_t}{\sigma_t}\right) \rho_t  - \rho_t + \sigma_t\right) \left(-\vct{u}_t \cdot \vct{\nu}\right) \D \vct{x}.
    \end{aligned}
    \end{equation}
\end{corollary}
\begin{proof}
    The result follows directly by combining \cref{eqn:transport_weak} (with $\rho_t$ replaced by $\sigma_t$) and \cref{lem:kl_bound_aux}.
\end{proof}
The above lemma signifies that the KL divergence is constant, up to effects due to in/outflow.
The first term amounts to the change of KL divergence on the inflow boundary, due to the inflowing mass.
The second and third term account for the KL divergence leaving the domain along the outflow boundary.
Note that a similar property is \emph{not} true for any norm equivalent to the $L^2$ norm, unless $\vct{u}$ is divergence-free.
An analog property does hold for the $L^1$ norm.
However, the $L^1$ norm is not associated to an inner product.
The key property of the KL divergence is that it is invariant under the transport equation's evolution operators while possessing an inner product structure that allows defining Galerkin projection methods.

\subsection{Error bounds of continuous Fisher-Rao Galerkin, in KL divergence}

\subsubsection*{Continuous Galerkin approximation}
In the presence of Galerkin approximation, the constancy of the KL divergence can not hold exactly.
But as we show next, the increase in KL divergence between the true and Fisher-Rao Galerkin solution is given by the Fisher-Rao inner product betwen the PDE residual and the current approximation error.

\begin{theorem}
    \label{thm:FRG_KL_bound}
    Let $\rho_t$ satisfy the transport equation \cref{eqn:transport_weak} with $(\vct{u}_t \equiv 0)$ on $\partial \pdedomain$ and let $\trf_t$ be the continuous Fisher-Rao Galerkin approximation of $\rho_t$.
    Then, if $\trf$ is positive, the KL divergence $\DivKL{\rho_t}{\trf_t}$ satisfies
    \begin{equation}
        \frac{\D}{\D t} \DivKL{\rho_t}{\trf_t} 
        = \inf \limits_{\tef \in \trs}\left \langle \dot{\trf}_t + \divergence\left(\trf_t \vct{u}_t\right), \rho_t - \tef \right \rangle_{\trf_t}
        \leq \left\|\dot{\trf}_t + \divergence\left(\trf_t \vct{u}_t\right)\right\|_{\trf_t} \left\|\rho_t - \hat{\trf}_t \right\|_{\trf_t}, 
    \end{equation}
    where $\hat{\trf}_t$ denotes the $\trf_t$-Fisher-Rao projection of $\rho_t$, onto $\trs$.
\end{theorem}
\begin{proof}
This follows from 
\cref{lem:kl_bound_aux} and the continuous Fisher-Rao Galerkin property \cref{eqn:FRgalerkin_transport}.
\end{proof}
\Cref{thm:FRG_KL_bound} shows that the KL divergence between the true solution $\rho_t$ and its Fisher-Rao Galerkin approximation $\trf_t$ grows slowly, if either the residual is small in Fisher-Rao norm or the true solution is (in Fisher-Rao norm) close to the approximation space.
Crucially, $\rho_t$ and $\trf_t$ need not be close for this to hold.

\subsubsection*{Discontinuous Galerkin approximation}
An analog result holds for the discontinuous variant.
\begin{theorem}
    \label{lem:thm:DFRG_KL_bound}
    Let $\rho_t > 0$ be a Lipschitz solution of the transport equation \cref{eqn:transport_weak} with $(\vct{u}_t \equiv 0)$ on $\partial \pdedomain$ and let $\trf_t > 0$ be the discontinuous Fisher-Rao Galerkin approximation of $\rho_t$.
    Denote as $\mathbb{K}^\circ_t$ the set of interfaces between two subdomains $\omega_k, \omega_l$ along which $\vct{u}_t$ points exclusively either from $\omega_k$ to $\omega_l$ or from $\omega_l$ to $\omega_k$.
    When integrating over $\Delta \in \mathbb{K}^\circ_t$, we denote as $\rho^+_t, \trf^+_t$ the interface values on the downstream element and as $\rho^-_t, \trf^-_t$ the interface values on the upstream element.
    Then, the KL divergence $\DivKL{\rho_t}{\trf_t}$ satisfies
    \begin{equation}
        \begin{aligned}
            \frac{\D}{\D t} \DivKL{\rho_t}{\trf_t} = \inf \limits_{\tef \in \trs}&
            \phantom{+} \sum \limits_{k} \left(- \int \limits_{\omega_k} \frac{\left(\rho_t - \tef\right) \dot{\trf}_t }{\trf_t}\D \vct{x} 
            + \int \limits_{\omega_k} \cnst{D}\left(\frac{ \left(\rho_t - \tef\right)}{\trf_t}\right) \trf_t \vct{u}_t\D \vct{x} 
            - \int \limits_{\partial \omega_k} \frac{\left(\rho_t - \tef\right)f_{\mathrm{up}}^{\omega_k, \trf_t, \vct{u}_t}}{\trf_t}\D \vct{x}\right)\\
            & + \sum \limits_{\Delta \in \mathbb{K}^{\circ}_{t}} \underbrace{\int \limits_{\Delta} \left(\log\left(\frac{\trf^+_t}{\trf_t^-}\right) - \frac{\trf^-_t}{\trf_t^+}  + 1 \right) \rho_t \left|\vct{u}_t \cdot \vct{\nu}\right| \D \vct{x}.}_{\frac{\D}{\D s} \left. \bigDivKL{\left.\trf_t^-\right|_{\Delta} - s\rho_t \left|\vct{u}_t \cdot \vct{\nu}\right|}{\left.\trf_t^+\right|_{\Delta} + s\rho_t \left|\vct{u}_t \cdot \vct{\nu}\right|} \right|_{s = 0}}
        \end{aligned}
    \end{equation}
\end{theorem}
Note that as in the continuous case, the first sum is small if either the true solution is well-represented by the Galerkin approximation or the PDE residual is small, in Fisher-Rao norm.
The additional, second term is small so long the the DG solution does not feature upward jumps when following the flow direction.
\begin{proof}
    We apply \cref{lem:kl_bound_aux} and use the definition of $f_{\mathrm{up}}$ to obtain
    \begin{align*}
        &\frac{\D}{\D t} \DivKL{\rho_t}{\trf_t} 
        = \sum \limits_{k} \frac{\D}{\D t} \bigDivKL{\left.\rho_t\right|_{\omega_k}}{\left.\trf_t\right|_{\omega_k}} \\
        &\phantom{\frac{\D}{\D t} \DivKL{\rho_t}{\trf_t}} =
        \sum \limits_{k}
        \left(
        \begin{aligned}
            &- \int \limits_{\omega_k} \frac{ \rho_t \dot{\trf}_t }{\trf_t}\D \vct{x} 
            + \int \limits_{\omega_k} \cnst{D}\left(\frac{ \rho_t}{\trf_t}\right) \trf_t \vct{u}_t\D \vct{x} 
            - \int \limits_{\partial \omega_k} \frac{\rho_t f_{\mathrm{up}}^{\omega_k, \trf_t, \vct{u}_t}}{\trf_t}\D \vct{x}+ \int \limits_{\omega_k} \dot{\trf}_t \D \vct{x} \\
            & \\
            &+ \int \limits_{\partial \omega_k} f_{\mathrm{up}}^{\omega_k, \trf_t, \vct{u}_t} \D \vct{x}+ \int \limits_{\partial \omega_k} \log\left(\frac{\rho_t}{\trf_t}\right) f_{\mathrm{up}}^{\omega_k, \rho_t, \vct{u}_t}  \D \vct{x}
              + \int \limits_{\partial \omega_k} \left(\frac{\rho_t}{\trf_t} - \frac{\trf_t}{\trf_t} \right) f_{\mathrm{up}}^{\omega_k, \trf_t, \vct{u}_t}\D \vct{x}
        \end{aligned}\right).
    \end{align*}
    Here, we abuse notation slightly and write $f_{\mathrm{up}}^{\omega_k, \trf_t, \vct{u}_t}, f_{\mathrm{up}}^{\omega_k, \rho_t, \vct{u}_t} \coloneqq \rho^{\mathrm{in}}_t \vct{u}_t \cdot \vct{\nu}$ on $\partial \omega_k \cap \inbound_t$.
    Using the fact that $\trf$ is a discontinuous Fisher-Rao Galerkin solution, we obtain for any $\tef \in \trs$, 
    \begin{equation}
        \label{eqn:dg_kl_bound_aux} 
        \frac{\D}{\D t} \DivKL{\rho_t}{\trf_t} 
        = \sum \limits_{k}
        \left(
        \begin{aligned}
            &- \int \limits_{\omega_k} \frac{\left(\rho_t - \tef\right) \dot{\trf}_t }{\trf_t}\D \vct{x} 
            + \int \limits_{\omega_k} \cnst{D}\left(\frac{ \left(\rho_t - \tef\right)}{\trf_t}\right) \trf_t \vct{u}_t\D \vct{x} 
            - \int \limits_{\partial \omega_k} \frac{\left(\rho_t - \tef\right)f_{\mathrm{up}}^{\omega_k, \trf_t, \vct{u}_t}}{\trf_t}\D \vct{x}\\
            &+ \int \limits_{\partial \omega_k} \log\left(\frac{\rho_t}{\trf_t}\right) f_{\mathrm{up}}^{\omega_k, \rho_t, \vct{u}_t}  \D \vct{x}
            + \int \limits_{\partial \omega_k} \left(\frac{\rho_t}{\trf_t} - \frac{\trf_t}{\trf_t} \right) f_{\mathrm{up}}^{\omega_k, \trf_t, \vct{u}_t}\D \vct{x}
        \end{aligned}\right)
    \end{equation}
    As noted above, we denote as $\mathbb{K}^\circ_t$ the set of interfaces between two subdomains $\omega_k, \omega_l$ along which $\vct{u}_t$ points exclusively either from $\omega_k$ to $\omega_l$ or from $\omega_l$ to $\omega_k$.
    When integrating over $\Delta \in \mathbb{K}^{\circ}_t $, we denote as $\rho^+_t, \trf^+_t$ the interface values on the downstream element and as $\rho^-_t, \trf^-_t$ those on the upstream element, resulting in
    \begin{equation*}
        \begin{aligned}
            &\sum \limits_{k} \left(\int \limits_{\partial \omega_k} \log\left(\frac{\rho_t}{\trf_t}\right) f_{\mathrm{up}}^{\omega_k, \rho_t, \vct{u}_t}  \D \vct{x}
            + \int \limits_{\partial \omega_k} \left(\frac{\rho_t}{\trf_t} - \frac{\trf_t}{\trf_t} \right) f_{\mathrm{up}}^{\omega_k, \trf_t, \vct{u}_t}\D \vct{x}\right) \\
            =&\sum \limits_{\Delta \in \mathbb{K}^{\circ}_{t}} \left(\int \limits_{\Delta} \left(\log\left(\frac{\rho_t^-}{\trf_t^-}\right) - \log\left(\frac{\rho_t^+}{\trf_t^+}\right)\right) \rho_t^{-} \left|\vct{u}_t \cdot \vct{\nu}\right|  \D \vct{x}
             + \int \limits_{\Delta} \left(\left(\frac{\rho_t^-}{\trf_t^-} - \frac{\trf_t^-}{\trf_t^-} \right) - \left(\frac{\rho_t^+}{\trf_t^+} - \frac{\trf_t^+}{\trf_t^+} \right) \right) \trf_t^{-} \left|\vct{u}_t \cdot \vct{\nu}\right| \D \vct{x}\right)\\
            &+  \int \limits_{\outbound_t} \left(\log\left(\frac{\rho_t}{\trf_t}\right) \rho_t  - \rho_t + \trf_t\right) \left(-\vct{u}_t \cdot \vct{\nu}\right) \D \vct{x}
            + \int \limits_{\inbound_t} \left(\log\left(\frac{\rho_t}{\trf_t}\right) \rho^{\mathrm{in}}_t           
            + \left(\frac{\rho_t}{\trf_t} - \frac{\trf_t}{\trf_t}\right) \trf_t^{\mathrm{in}}\right) \left(- \vct{u}_t \cdot \vct{\nu}\right)\D \vct{x}.
        \end{aligned}
    \end{equation*}
    We simplify the summation over edge terms as  
    \begin{equation*}
        \begin{aligned}
            &\sum \limits_{\Delta \in \mathbb{K}^{\circ}_{t}} \left(\int \limits_{\Delta} \left(\log\left(\frac{\rho_t^-}{\trf_t^-}\right) - \log\left(\frac{\rho_t^+}{\trf_t^+}\right)\right) \rho_t^{-} \left|\vct{u}_t \cdot \vct{\nu}\right|  \D \vct{x}
             + \int \limits_{\Delta} \left(\left(\frac{\rho_t^-}{\trf_t^-} - \frac{\trf_t^-}{\trf_t^-} \right) - \left(\frac{\rho_t^+}{\trf_t^+} - \frac{\trf_t^+}{\trf_t^+} \right) \right) \trf_t^{-} \left|\vct{u}_t \cdot \vct{\nu}\right|\D \vct{x}\right)\\
            = &\sum \limits_{\Delta \in \mathbb{K}^{\circ}_{t}} \int \limits_{\Delta} \left(\log\left(\frac{\trf^+_t}{\trf_t^-}\right) \rho_t 
             + \left(\frac{\rho_t}{\trf_t^-}  - \frac{\rho_t}{\trf_t^+} \right) \trf_t^{-} \right)\left|\vct{u}_t \cdot \vct{\nu}\right| \D \vct{x}
            = \sum \limits_{\Delta \in \mathbb{K}^{\circ}_{t}} \int \limits_{\Delta} \left(\log\left(\frac{\trf^+_t}{\trf_t^-}\right) - \frac{\trf^-_t}{\trf_t^+}  + 1 \right) \rho_t \left|\vct{u}_t \cdot \vct{\nu}\right| \D \vct{x}\\
           =& \sum \limits_{\Delta \in \mathbb{K}^{\circ}_{t}} \int \limits_{\Delta} \frac{\D}{\D s} \left.\log\left(\frac{\trf^-_t - s}{\trf_t^+ + s}\right) \left(\trf_t^- - s \right) - \left(\trf_t^- - s\right) + \left(\trf_t^+ + s\right) \right|_{s = 0} \rho_t \left|\vct{u}_t \cdot \vct{\nu}\right| \D \vct{x}\\
           = & \sum \limits_{\Delta \in \mathbb{K}^{\circ}_{t}} \frac{\D}{\D s} \left. \BigDivKL{\left.\trf_t^-\right|_{\Delta} - s\rho_t \left|\vct{u}_t \cdot \vct{\nu}\right|}{\left.\trf_t^+\right|_{\Delta} + s\rho_t \left|\vct{u}_t \cdot \vct{\nu}\right|} \right|_{s = 0}.
        \end{aligned}
    \end{equation*}
    Plugging this result into \cref{eqn:dg_kl_bound_aux}, using $\vct{u}_t \equiv0$ on $\partial \pdedomain$, and taking the infimum over $\tef \in \trs$ yields the result.
\end{proof}

\section{Implementation and numerical experiments}

\subsection{Implementation}

\subsection*{Deriving the numerical scheme} 

The boundary term in equation \cref{eqn:DFRgalerkin_transport} is replaced with a numerical flux term $\hat{f}$.
For all $1 \leq j \leq m$,
\begin{equation}
    \label{eqn:DFRgalerkin_transport_flux}
    \sum_{j = 1}^{m}  \dot{\trc}_j \int \limits_{\omega_{o_i}} \frac{\phi_j \phi_i}{\sum_{k = 1}^m \trc_k \phi_k} \D \vct{x} 
    -  \trc_j \int \limits_{\omega_{o_i}}\left[\cnst{D} \frac{\phi_i}{{\sum_{k = 1}^m \trc_k \phi_k}} \right] \phi_j \vct{u} \D \vct{x} 
    + \int \limits_{\partial \omega_{o_i}}\frac{\phi_i  }{{\sum_{k = 1}^m \trc_k \phi_k}} f_{\text{up}}^{\omega_{k},\hat{\rho},\vct{u}_{t}} \D \vct{x} 
    = 0 .
\end{equation}
This semidiscretization is expressed in matrix form as
\begin{equation}
    \mtx{M}^{\trf} \vct{\dot{\trc}} + \mtx{K}^{\trf}\vct{\trc} + \vct{g}^{\trf} = 0,
\end{equation}
and rearranged to obtain
\begin{equation}
    \label{eqn:DFRgalerkin_transport_matrix}
    \vct{\dot{\trc}} = 
    - \left( \mtx{M}^{\trf}  \right)^{-1} \left( \mtx{K}^{\trf} + \vct{g}^{\trf} \right),
\end{equation}
which is solvable by standard time integratorj.
We express the mass matrix, stiffness matrix and numerical flux vector as 
\begin{equation*}
    \mtx{M}^{\trf} = 
    \left(
        \int \limits_{\omega_{o_i}} \frac{\phi_j \phi_i}{\sum_{k = 1}^m \trc_k \phi_k} \D \vct{x}
    \right)_{1 \leq i, j \leq m} ,\quad
    \mtx{K}^{\trf} = 
    \left(  
        - \int \limits_{\omega_{o_i}}\left(\cnst{D} \frac{\phi_i}{{\sum_{k = 1}^m \trc_k \phi_k}} \right) \phi_j \vct{u} \D \vct{x}  
    \right)_{1 \leq i, j \leq m}
\end{equation*}
\begin{equation*}
    \vct{g}^{\trf} = 
    \left( 
        \int \limits_{\partial \omega_{o_i}}\frac{\phi_i  }{{\sum_{k = 1}^m \trc_k \phi_k}} f_{\text{up}}^{\omega_{k},\hat{\rho},\vct{u}_{t}} \D \vct{x} 
    \right)_{1 \leq i, j \leq m}.
\end{equation*}
To simplify the computation of the stiffness term, we move the  divergence to the velocity term through integration by parts and expand it to obtain
\begin{equation}
    \mtx{K}^{\trf} = 
    \left(  
        \int \limits_{\omega_{o_i}} \frac{\phi_{i} D \phi_{j}}{\sum_{k = 1}^{m} \trc_k \phi_k} \cdot \vct{u}  \D \vct{x}
        + \int \limits_{\omega_{o_i}} \frac{\phi_{i}  \phi_{j}}{\sum_{k = 1}^{m} \trc_k \phi_k} D \vct{u} \D \vct{x}
        - \int \limits_{\partial \omega_{o_i}} \frac{\phi_{i} \phi_{j}}{\sum_{k = 1}^{m} \trc_k \phi_k} \vct{u} \cdot \vct{\nu}  \D \vct{x}
    \right)_{1 \leq i, j \leq m}.
\end{equation}
Expressing the velocity in terms of the basis $\{ \phi_{l} \}_{ 1 \leq l \leq m}$ as $\vct{u} = \sum_{l=1}^{m} \vct{u}_{l} \phi_{l}$, the stiffness matrix becomes
\begin{equation}
    \mtx{K}^{\trf} = 
     \left(  \sum_{l=1}^{m}
        \int \limits_{\omega_{o_i}} \frac{\phi_{i}  \phi_{l}}{\sum_{k = 1}^m \trc_k \phi_{k}}D \phi_{j} \cdot \vct{u}_{l}   \D \vct{x}
        + \int \limits_{\omega_{o_i}} \frac{ \phi_{i}  \phi_{j}}{\sum_{k = 1}^m \trc_k \phi_k} D \phi_{l} \cdot \vct{u}_{l}    \D \vct{x}
        - \int \limits_{\partial \omega_{o_i}} \frac{\phi_{i} \phi_{j} \phi_{l}}{\sum_{k = 1}^m \trc_k \phi_k} \vct{u}_{l} \cdot \vct{\nu}  \D \vct{x}
    \right)_{1 \leq i, j \leq m}.
\end{equation}

\subsection*{Implementation specifics}
All experiments described below were conducted on an Apple M2 8-core CPU running at 3.49GHz with 24 GB of RAM.
The code, along with the examples discussed below, are available at https://github.com/brookeyob/FRG.
We implemented the discontinuous Fisher-Rao Galerkin method in Julia, and tested cases in one and two dimensions.
We use a uniform mesh, which was quadrilateral in the 2D case.
The basis function set $\{ \phi_{j} \}_{1 \leq j \leq m}$ was defined by the Lagrange polynomials.
In 1D, we specify the Lagrange polynomial of order $p$ by points at $\{ x_{k} \}_{1 \leq k \leq p}$ as
\begin{equation}
    L_{i}(x) = \begin{cases}
        \prod_{\substack{ k = 0, \\ k \neq i } } \frac{x - x_{k}}{x_{i} - x_{k}},  & x \in \omega_{o_{i}} \\
        0, & \text{else}
    \end{cases}.
\end{equation}
We extend this to higher dimensions by taking products of the polynomials in 1D.
We choose the ${x_{k}}$ to be equidistant.
The mass matrix $\mtx{M}^{\trf}$ and stiffness matrix $\mtx{K}^{\trf}$ are block diagonal.
The block size are $p+1$ in 1D and $(p+1)^2$ in 2D, for basis functions with polynomial order $p$.
The block is determined by the number of basis functions in each $\omega_{o_{i}}$, for $1 \leq i \leq m$.
We perform numerical integration using Clenshaw–Curtis quadrature.
The Clenshaw–Curtis weights are positive meaning that positive integrands always yield positive integrals.
The spatial semidiscritization in equation \cref{eqn:DFRgalerkin_transport_matrix} was combined with an explicit third-order strong stability preserving Runge-Kutta (SSPRK3) time integration scheme \cite{isherwood2018strong} to obtain a solution.

\subsection{Numerical experiments}
\subsection*{Explored scenarios}

We demonstrate the utility of (discontinuous) Fisher-Rao Galerkin methods through numerical examples.
Specifically, we highlight how the DFRG method preserves positivity and maintains invariance properties in both 1D and 2D.
The invariance to reparametrization becomes crucial when resolving solutions with a changing scale of density variations, such as in cycles of compression and expansion.
Large density variations increase the likelihood of negative values, as small values interact with large fluxes.
Higher-order schemes reduce dispersion and dissipation errors, achieving higher accuracy with fewer degrees of freedom \cite{Ainsworth2004}, but they often produce spurious oscillations that introduce negative values in rarefied regions near steep gradients.
To ensure monotonicity at discontinuities, first-order schemes are often used, but they often add excessive diffusion to the solution \cite{Krivodonova2004}.
The DFRG method maintains positivity without sacrificing accuracy.
Rather, it achieves accuracy with respect to a different measure of error, the KL divergence instead of the $L^{2}$ norm.
% The examples shown here are toy problems that are not meant to be physically realistic, but rather to highlight the properties of a DFRG method.
To simplify the examples, we use constant in time velocity fields and periodic boundary conditions over the domain $[0,1]^d$.

\begin{figure}[htb]

   \centering

   \scalebox{1.0}{
      \includegraphics{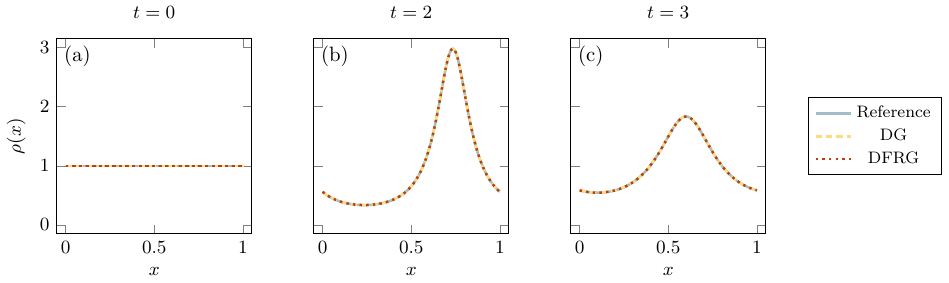}
   }

   \caption{ Density profile for \cref{ex:1D-mild-compression}, a 1D sinusoidal problem with polynomial order $p=1$, constant initial density, and fixed velocity $u(x) = sin(2 \pi x) + 2.0$.
   The solution is obtained using the DG and DFRG methods with CFL = 0.1875 and $m=256$ cells.
   The profiles are shown at (a) $t=0$, (b) $t=2$, and (c) $t=3$.} 

   \label{fig:sinusoidal1_1D_density}

\end{figure}

\subsection*{Quantifying performance}
We quantify the performance of the DFRG method by measuring the error convergence under grid refinement.
We measure the error with respect to the $L^{1}$, $L^{2}$ norms, and KL divergence.
The $L^{2}$ norm error increases quadratically with density deviation and is dominated by high-density regions.
In contrast, the $L^{1}$ norm and KL divergence errors increase linearly with density deviation, making them more sensitive to regions of moderate density.
The KL divergence proves particularly useful for comparing relative deviations.

The experiments demonstrate the benefits of minimizing error with respect to the KL divergence, which the DFRG method uses, as opposed to the $L^{2}$ norm used in the DG method or even the $L^{1}$ norms.
We compare the DFRG method to a positivity preserving ($+$)-limiter method \cite{zhang2010positivity}.
The ($+$)-limiter method adjusts cells with negative values in a subdomain $\omega_{k}$, but with positive overall mass over $\omega_{k}$, by pulling the values to the mean of the cell until all values in the cell become positive.
If the mass of the cell is negative, the method sets the density of the cell to a small value $\epsilon = 10^{-15}$.
The ($+$)-limiter enables the computation of KL divergence errors, as the solution is always positive.
However, if any cells have negative mass, then the solution loses mass conservation, and the mass of the solution may drift over time depending on how often the positivity-preserving mechanism activates.
The DFRG method is locally mass conserving, and numerical results show that mass deviation over time remains on the order of machine precision.

\begin{figure}[htb]
   \centering

   \scalebox{1.0}{
      \includegraphics{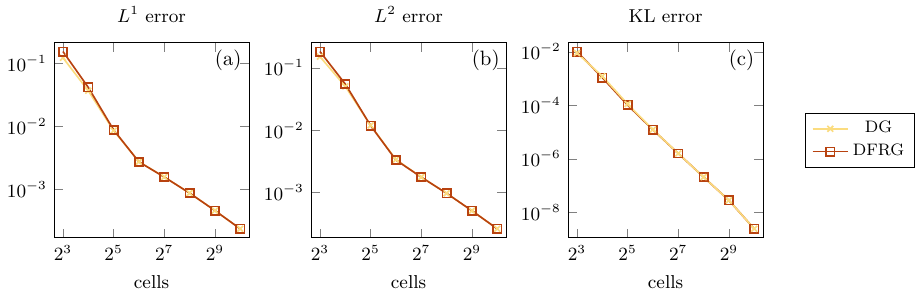}
   }

   \caption{ Error convergence under grid refinement for \cref{ex:1D-mild-compression}, a 1D sinusoidal problem with polynomial order $p=1$, constant initial density, and fixed velocity $u(x) = sin(2 \pi x) + 2.0$. The mean of the (a) $L^{1}$, (b) $L^{2}$, and (c) $KL$ errors with respect to the exact solution is measured at points spaced out by $\Delta t=0.01$ up to time $T=3$, for the DG, DG $+$-limiter and DFRG methods.}

   \label{fig:sinusoidal1_1D_error}

\end{figure}

\begin{table}[H]
    \centering

        \pgfplotstableset{  col sep=comma,
                            columns={   Discretization,
                                        L1_DG_error,
                                        L1_DG+limiter_error,
                                        L1_DFRG_error,
                                        L2_DG_error,
                                        L2_DG+limiter_error,
                                        L2_DFRG_error,
                                        KL_DG+limiter_error,
                                        KL_DFRG_error
                                        },
                            font=\footnotesize
                        }
        \begin{filecontents*}{error_csv/average_error_norm_method_disc_sinusoidal1_d_1_p_1_T_3-0_n_t_1_FR_weight_ref_false_is_exact_true/combined_error_values.csv}
Discretization,L1_DG_error,L1_DG+limiter_error,L1_DFRG_error,L2_DG_error,L2_DG+limiter_error,L2_DFRG_error,KL_DG_error,KL_DG+limiter_error,KL_DFRG_error
8,0.12281282231433484,0.12281282231433484,0.15423346359413206,0.15120101928721824,0.15120101928721824,0.1826385846616552,0.0090148841985021,0.0090148841985021,0.010254900556875805
16,0.03682511883892915,0.03682511883892915,0.04158298935181376,0.0504493427434203,0.0504493427434203,0.0550138142573538,0.0012242915415209845,0.0012242915415209845,0.0010876463626304455
32,0.008564873992544287,0.008564873992544287,0.0088116938049494,0.01196644664687504,0.01196644664687504,0.011874584448450656,0.0001123807829880716,0.0001123807829880716,0.00010214476319700333
64,0.0027554898387608623,0.0027554898387608623,0.0027346938289340944,0.003397202901237084,0.003397202901237084,0.00333329969378274,1.2421459107418696e-5,1.2421459107418696e-5,1.2160684531993705e-5
128,0.0015736388065990458,0.0015736388065990458,0.0015707951104049623,0.0017766995606208272,0.0017766995606208272,0.0017740527251391865,1.589984083549641e-6,1.589984083549641e-6,1.5850114196610644e-6
256,0.00087542078165766,0.00087542078165766,0.00087528874042156,0.0009770363827056104,0.0009770363827056104,0.0009769484548319662,2.0859075483612771e-7,2.0859075483612771e-7,2.0846167407317027e-7
512,0.00046186521584321163,0.00046186521584321163,0.00046185943568449766,0.0005142278505726618,0.0005142278505726618,0.0005142225499430204,2.801301168052818e-8,2.801301168052818e-8,2.8011097362411503e-8
1024,0.00023709070641575138,0.00023709070641575138,0.00023709042680069713,0.00026382490865781556,0.00026382490865781556,0.0002638242337089625,2.4277347302006414e-9,2.4277347302006414e-9,2.427687377925049e-9
\end{filecontents*}
\pgfplotstabletypeset[
            every head row/.style={
                before row= { \toprule
                & \multicolumn{3}{c}{L$^{1}$ error}
                & \multicolumn{3}{c}{L$^{2}$ error}
                & \multicolumn{2}{c}{KL error} \\
                },
                after row=\midrule},
            every last row/.style={
                after row=\bottomrule},
            columns/Discretization/.style={column name=$\#$ cells},
            columns/L1_DG_error/.style={column name=DG, fixed, sci , sci 10e, sci zerofill, precision=2},
            columns/L1_DG+limiter_error/.style={column name=$+$-limiter, fixed, sci , sci 10e, sci zerofill, precision=2},
            columns/L1_DFRG_error/.style={column name=DFRG, fixed, sci , sci 10e, sci zerofill, precision=2},
            columns/L2_DG_error/.style={column name=DG, fixed, sci , sci 10e, sci zerofill, precision=2},
            columns/L2_DG+limiter_error/.style={column name=$+$-limiter, fixed, sci , sci 10e, sci zerofill, precision=2},
            columns/L2_DFRG_error/.style={column name=DFRG, fixed, sci , sci 10e, sci zerofill, precision=2},
            columns/KL_DG+limiter_error/.style={column name=$+$-limiter, fixed, sci , sci 10e, sci zerofill, precision=2},
            columns/KL_DFRG_error/.style={column name=DFRG, fixed, sci , sci 10e, sci zerofill, precision=2},
            columns/date/.style={string type},
        ]{error_csv/average_error_norm_method_disc_sinusoidal1_d_1_p_1_T_3-0_n_t_1_FR_weight_ref_false_is_exact_true/combined_error_values.csv}
    
        \caption{Error convergence under grid refinement for \cref{ex:1D-mild-compression}, a 1D sinusoidal problem with polynomial order $p=1$, constant initial density, and fixed velocity $u(x) = sin(2 \pi x) + 2$ .The mean of the $L^{1}$, $L^{2}$, and $KL$ errors with respect to the exact solution is measured at points spaced out by $\Delta t = 0.01$ up to time $T=3$ , for the DG, DG $+$-limiter, and DFRG methods.}
        \label{table:sinusoidal1_1D_error}
\end{table}

\subsection*{Sinusoidal problem}
We illustrate how the DFRG method handles compression and expansion cycles with large differences in scale by considering a 1D problem.
The problem uses a constant initial density and a sinusoidal velocity profile $u(x) = \sin(2\pi x) + 1 + a$, where $a > 0$.
This velocity profile causes cycles of mass accumulating near areas of relatively low velocity and then dispersing due to the bulk motion of the fluid.
The parameter $a$ controls the minimum velocity; smaller values of $a$ lead to larger differences in scale between the compressed and rarified regions as more mass accumulates in areas of low velocity.
%Similar patterns are observed in (pressure) waves, where the medium is compressed and rarified as the wave propagates.

\begin{figure}[htb]
   \centering

      \includegraphics{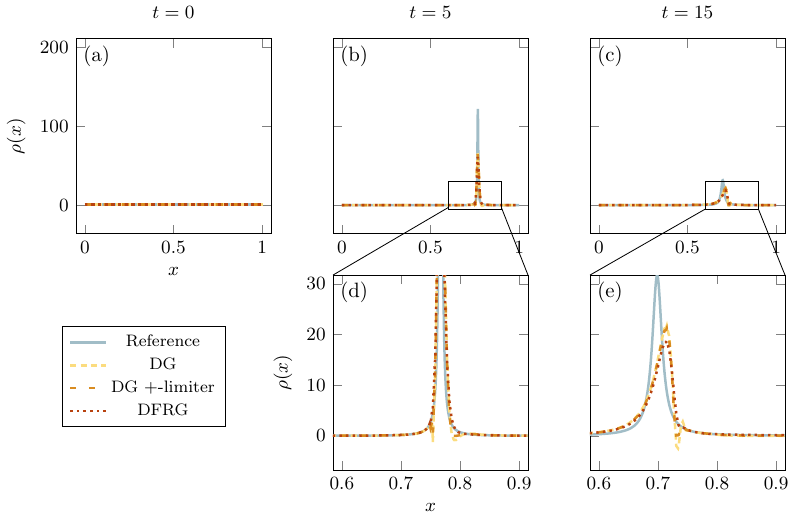}
      \includegraphics{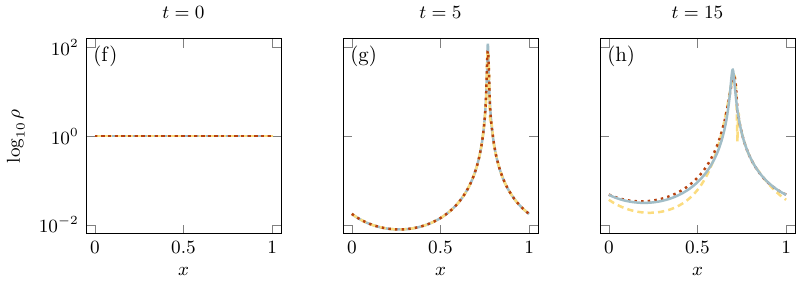}

   \caption{ Density profile for \cref{ex:1D-extreme-compression}, a 1D sinusoidal problem with polynomial order $p=1$, constant initial density, and fixed velocity $u(x) = sin(2 \pi x) + 1.01$.
   The solution is obtained using the DG and DFRG methods with CFL = 0.125, $m=256$ cells for plots (a-e), and $m=512$ cells for plots (f-h).
   The profiles are shown at (a,f) $t=0$, (b,d,g) $t=5$, and (c,e,h) $t=15$. } 

   \label{fig:sinusoidal2_1D_density_log_combo}

\end{figure}

\vspace{-5pt}

% \subsection*{Example 6.2.1: Mild compression} 
\begin{example}[Mild compression]
\label{ex:1D-mild-compression}
The example in \cref{fig:sinusoidal1_1D_density} uses a sinusoidal velocity profile with $a = 1$, which causes a relatively mild compression.
This example does not lead to a loss of positivity for the DG method and illustrates how the DFRG method performs when the positivity-preserving mechanism remains inactive.
\cref{fig:sinusoidal1_1D_error,table:sinusoidal1_1D_error} illustrates the grid convergence of the DG and DFRG methods under these conditions.
The errors with respect to the $L^{1}$, $L^{2}$ norms, and KL divergence converge at approximately the same rate.
When the positivity-preserving mechanism is inactive, the DFRG weighting does not significantly affect the solution's accuracy.
\end{example}

%\subsection*{Example 6.2.2: Extreme compression} 
\begin{example}[Extreme compression]
\label{ex:1D-extreme-compression}
The example in \cref{fig:sinusoidal2_1D_density_log_combo} uses a sinusoidal velocity profile with $a = 0.01$.
This small $a$ value causes a relatively extreme compression, creating a much larger difference in scale between the compressed and rarified regions.
This example is extreme enough to cause the DG method to lose positivity.
\cref{fig:sinusoidal2_1D_error_time_combo,table:sinusoidal2_1D_error} show the grid convergence of the DG and DFRG methods when the positivity-preserving mechanism activates.
The errors with respect to the $L^{1}$ and $L^{2}$ norms converge at approximately the same rate for the DG, ($+$)-limiter, and DFRG methods.
However, the error with respect to the KL divergence diverges in all methods except DFRG.

\begin{figure}[H]

   \centering

   \includegraphics{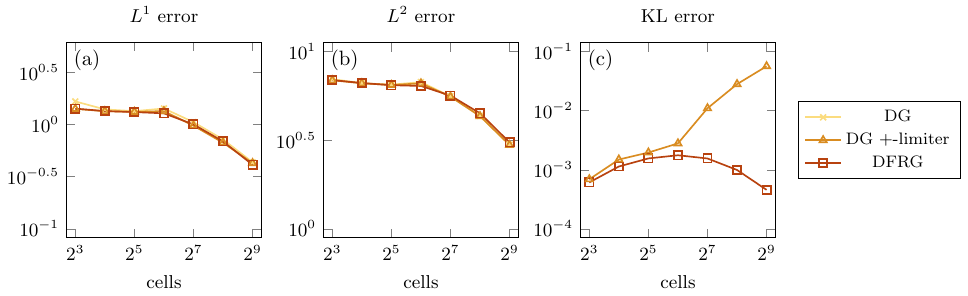}
   \includegraphics{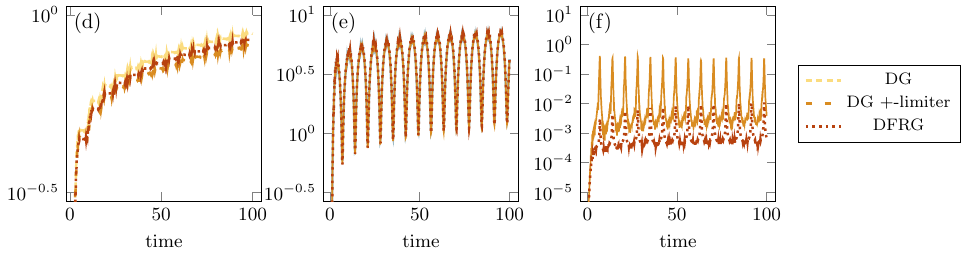}

   \caption{  Error convergence (a-c) and error over time (d-f) for \cref{ex:1D-extreme-compression}, a 1D sinusoidal problem with polynomial order $p=1$, constant initial density, and fixed velocity $u(x) = sin(2 \pi x) + 1.01$.
   The (a,d) $L^{1}$, (b,e) $L^{2}$ norms, and (c,f) $KL$ divergence with respect to the exact solution at points spaced out by $\Delta t=0.1$ up to time $T=100$ is shown for the DG, DG $+$-limiter and DFRG methods.
   The error convergence plots (a-c) show the average error over time, and the error over time plots (d-f), are for CFL = 0.1875 and $m=256$ cells.
   The KL error for DG is infinite, and thus not shown.}

   \label{fig:sinusoidal2_1D_error_time_combo}

\end{figure}

\begin{table}[htbp]
    \centering

        \pgfplotstableset{  col sep=comma,
                            columns={   Discretization,
                                        L1_DG_error,
                                        L1_DG+limiter_error,
                                        L1_DFRG_error,
                                        L2_DG_error,
                                        L2_DG+limiter_error,
                                        L2_DFRG_error,
                                        KL_DG+limiter_error,
                                        KL_DFRG_error
                                        },
                            font=\footnotesize
                        }
        \begin{filecontents*}{error_csv/average_error_norm_method_disc_sinusoidal2_d_1_p_1_T_100-0_n_t_10_FR_weight_ref_true_is_exact_true/combined_error_values.csv}
Discretization,L1_DG_error,L1_DG+limiter_error,L1_DFRG_error,L2_DG_error,L2_DG+limiter_error,L2_DFRG_error,KL_DG_error,KL_DG+limiter_error,KL_DFRG_error
8,1.6627381360623912,1.4101775362912694,1.4184796787452385,6.991972821690382,6.88546265,6.902110769450793,0.000688718,0.00068876,0.000610076
16,1.394385523,1.3396510423598613,1.343088670841573,6.627774475859524,6.621224444727125,6.641662597276366,0.001501226,0.001487503,0.001133126
32,1.3409602552907574,1.320854407367436,1.3144704584104039,6.505159962094212,6.475433808756835,6.458476708387218,0.002169247,0.001952638,0.001552822
64,1.4239489106783851,1.3428845101582854,1.277951120427126,6.722841484851483,6.612060596420488,6.393694219316405,0.003779513,0.002788113,0.00175609
128,1.0522652680439721,0.9740744601802013,1.0010601950975084,5.588437186910235,5.565875517882366,5.640654337018211,0.007870467,0.010957605400664948,0.001549629
256,0.718988049,0.6684183080865932,0.6882314864348217,4.335191749747758,4.317610852,4.484442968,0.018100525461702335,0.028045855962969674,0.000989884
512,0.43422088159042815,0.42597677004770185,0.4107452868299748,2.9824773315201765,2.968870996,3.0730475926334746,0.026640110804557207,0.056370792,0.000456033
\end{filecontents*}
\pgfplotstabletypeset[
            every head row/.style={
                before row= { \toprule
                & \multicolumn{3}{c}{L$^{1}$ error}
                & \multicolumn{3}{c}{L$^{2}$ error}
                & \multicolumn{2}{c}{KL error} \\
                },
                after row=\midrule},
            every last row/.style={
                after row=\bottomrule},
            columns/Discretization/.style={column name=$\#$ cells},
            columns/L1_DG_error/.style={column name=DG, fixed, sci , sci 10e, sci zerofill, precision=2},
            columns/L1_DG+limiter_error/.style={column name=$+$-limiter, fixed, sci , sci 10e, sci zerofill, precision=2},
            columns/L1_DFRG_error/.style={column name=DFRG, fixed, sci , sci 10e, sci zerofill, precision=2},
            columns/L2_DG_error/.style={column name=DG, fixed, sci , sci 10e, sci zerofill, precision=2},
            columns/L2_DG+limiter_error/.style={column name=$+$-limiter, fixed, sci , sci 10e, sci zerofill, precision=2},
            columns/L2_DFRG_error/.style={column name=DFRG, fixed, sci , sci 10e, sci zerofill, precision=2},
            columns/KL_DG+limiter_error/.style={column name=$+$-limiter, fixed, sci , sci 10e, sci zerofill, precision=2},
            columns/KL_DFRG_error/.style={column name=DFRG, fixed, sci , sci 10e, sci zerofill, precision=2},
            columns/date/.style={string type},
        ]{error_csv/average_error_norm_method_disc_sinusoidal2_d_1_p_1_T_100-0_n_t_10_FR_weight_ref_true_is_exact_true/combined_error_values.csv}
    
        \caption{Error convergence for \cref{ex:1D-extreme-compression}, a 1D sinusoidal problem with polynomial order $p=1$, constant initial density, and fixed velocity $u(x) = sin(2 \pi x) + 1.01$ . 
        The error is measured as the mean of the $L^{1}$, $L^{2}$ norms, and $KL$ divergence with respect to the exact solution at points spaced out by $\Delta t = 0.1$ up to time $T=100$ , for the DG, DG $+$-limiter, and DFRG methods.}
        \label{table:sinusoidal2_1D_error}
\end{table}

The KL divergence, which is more sensitive to relative differences in the solution, demonstrates the superior performance of the DFRG method compared to the DG method with the ($+$)-limiter in cases with large density variations.
\cref{fig:sinusoidal2_1D_error_time_combo} shows the error over time for the DG, ($+$)-limiter, and DFRG methods with respect to the $L^{1}$, $L^{2}$ norms, and KL divergence.
The error increases during compression periods and decreases during expansion periods.
The $L^{2}$ norm, which emphasizes large absolute deviations, disproportionately emphasizes errors in the compressed regions while ignoring small deviations in the rarified ones.
\cref{fig:sinusoidal2_1D_density_log_combo} also presents the log density plot for $m = 512$ elements, which does not require the ($+$)-limiter.
After a few compression and expansion cycles, the DG solution exhibits a larger relative error in the rarified regions compared to the compressed regions, whereas the DFRG solution maintains better accuracy.
The $L^{2}$ norm, which measures absolute differences, fails to capture this discrepancy, while the KL divergence does.
This shows the advantage of the DFRG method when relative accuracy in rarified regions is crucial.
This may help predict, for instance, whether chemical reactions relying on a small amount of a reactant occur.
\end{example}

%\subsection*{Example 6.2.3: Bump advection}
\begin{example}[Bump advection]
\label{ex:1D-bump}
Steep density gradients often lead to a loss of positivity due to dispersion error, which causes oscillations.
When the density approaches zero at the interface, these oscillations can produce negative density values.
To demonstrate this, we consider a 1D bump advection problem, as shown in \cref{fig:bump_1D_density_log_combo}, where a sigmoid function advects with a constant velocity.
The velocity is $u(x) = 1$, and the initial density is
\begin{figure}[h]
   \centering

   \scalebox{0.97}{
      \includegraphics{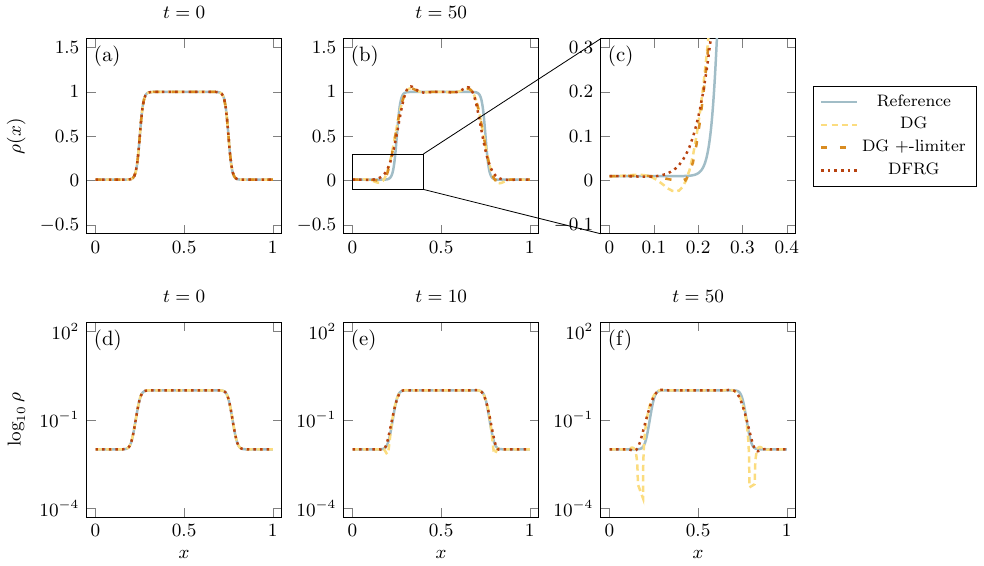}
      }

      \caption{   Density profile for \cref{ex:1D-bump},a 1D bump problem with polynomial order $p=1$, initial density equation \cref{eqn:bump_1D_density}, $b = 0.01$, $\mu = 0.5$, $k = 100$, and constant velocity $u(x) = 1$.
      The solution is obtained using the DG and DFRG methods with CFL = 0.0625, $m=128$ cells for plots (a-c), and $m=256$ cells for plots (d-f).
      Plots (d-f) are on a log scale.
      At $m=256$, DG remains positive and thus coincides with DG +-limiter
      The profiles are shown at (a,d) $t=0$, (e) $t=10$, and (b,c,f) $t=50$.} 
  
      \label{fig:bump_1D_density_log_combo}
      
  \end{figure}

\begin{equation}
   \rho(x,0) = 
   \begin{cases}
      (1-b) S(k(x-\mu)) + b & \text{if } 0 \leq x \leq 0.5, \\
      (b-1) S(k(x+\mu-1)) + 1  & \text{if } 0.5 \leq x \leq 1, \\
      b & \text{otherwise},
   \end{cases}
   \label{eqn:bump_1D_density}
\end{equation}
where $S$ is the standard logistic function, $b = 0.01$, $\mu = 0.5$, and $k = 100$.
\cref{fig:bump_1D_density_log_combo} highlights the oscillations near the steep gradient, where the DG method generates negative values in the low-density region, while the DFRG method preserves positivity.
\cref{fig:bump_1D_error,table:bump_1D_error} illustrates the grid convergence of the DG and DFRG methods for the bump problem.
As in the previous examples, the errors with respect to the $L^{1}$ and $L^{2}$ norms converge at approximately the same rate for all methods, while the KL divergence converges more quickly for the DFRG method.
\cref{fig:bump_1D_density_log_combo} also shows the log density plot for $m = 256$ elements, which does not require the ($+$)-limiter.
At $t=50$, the density near the interface is two orders of magnitude lower for the DG solution compared to the true solution.
The DFRG solution preserves the same order of magnitude throughout.
Instead, the DG method produces highly inaccurate results in rarefied regions near large density variations.
This effect resembles what we observe in the extreme compression example and demonstrates that it is not an isolated phenomenon.
\end{example}

\begin{figure}[H]
   \centering

   \scalebox{0.95}{
      \includegraphics{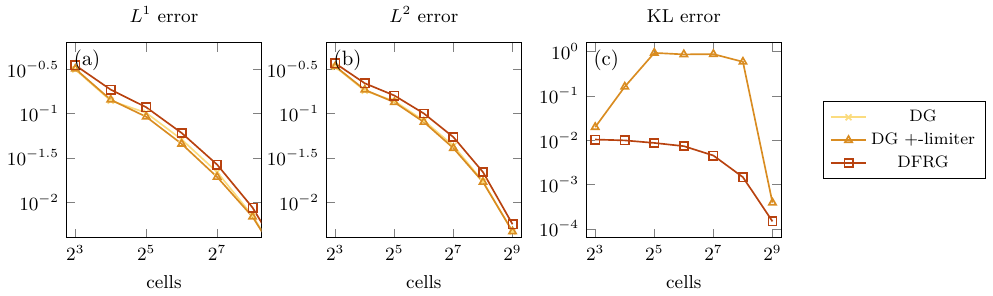}
   }

   \caption{   Error convergence for \cref{ex:1D-bump}, a 1D bump problem with polynomial order $p=1$, initial density equation \cref{eqn:bump_1D_density}, $b = 0.01$, $\mu = 0.5$, $k = 100$, and constant velocity $u(x) = 1$.
   The error is measured as the average of the (a) $L^{1}$, (b) $L^{2}$ norms, and (c) $KL$ divergence with respect to the exact solution at points spaced out by $\Delta t=0.01$ up to time $T=5$ , for the DG, DG $+$-limiter, and DFRG methods.  }

   \label{fig:bump_1D_error}

\end{figure}

\vspace{-10pt}

\begin{table}[htbp]
    \centering

        \pgfplotstableset{  col sep=comma,
                            columns={   Discretization,
                                        L1_DG_error,
                                        L1_DG+limiter_error,
                                        L1_DFRG_error,
                                        L2_DG_error,
                                        L2_DG+limiter_error,
                                        L2_DFRG_error,
                                        KL_DG+limiter_error,
                                        KL_DFRG_error
                                        },
                            font=\footnotesize
                        }
        \begin{filecontents*}{error_csv/average_error_norm_method_disc_bump_d_1_p_1_T_50-0_n_t_10_FR_weight_ref_true_is_exact_true/combined_error_values.csv}
Discretization,L1_DG_error,L1_DG+limiter_error,L1_DFRG_error,L2_DG_error,L2_DG+limiter_error,L2_DFRG_error,KL_DG_error,KL_DG+limiter_error,KL_DFRG_error
8,0.31347485768690164,0.3189066750804908,0.350215080758913,0.33572455934313405,0.3405588817159153,0.36799258974665183,0.03099408026600516,0.01970581165081473,0.010330033772315858
16,0.13940670299715974,0.14280904604789785,0.18560914211536084,0.18057469416594193,0.1850554633094153,0.21909195967290368,0.19142035826815554,0.16031820239614805,0.009880910623779344
32,0.10169368066269144,0.09175316923099792,0.11755065204808399,0.13813700340015614,0.13385962581199878,0.15943218711833418,0.24732385625407316,0.931168922168368,0.008658218538243343
64,0.05091341682781952,0.045427579323672626,0.06027919728258792,0.08308099687546242,0.08005492778014962,0.09998018805138156,0.38791208036602937,0.8619185587569959,0.007325995606220091
128,0.02181258135804737,0.019203115187906076,0.026497993736904174,0.04319813690046376,0.04060592386947087,0.054091067809015216,0.38015391622514255,0.8739839715579252,0.004497828328115726
256,0.0070307558647100515,0.006826044042312693,0.008672419802350126,0.01713281274047825,0.01683159054733141,0.0218114436743485,0.2707826018039578,0.5858777187262222,0.0014606277268207203
512,0.0015728924768841053,0.0015728924768841053,0.0018708280106994454,0.0046582903444338745,0.0046582903444338745,0.005627917512487958,0.0003884149469697953,0.0003884149469697953,0.00014718484977262262
\end{filecontents*}
\pgfplotstabletypeset[
            every head row/.style={
                before row= { \toprule
                & \multicolumn{3}{c}{L$^{1}$ error}
                & \multicolumn{3}{c}{L$^{2}$ error}
                & \multicolumn{2}{c}{KL error} \\
                },
                after row=\midrule},
            every last row/.style={
                after row=\bottomrule},
            columns/Discretization/.style={column name=$\#$ cells},
            columns/L1_DG_error/.style={column name=DG, fixed, sci , sci 10e, sci zerofill, precision=2},
            columns/L1_DG+limiter_error/.style={column name=$+$-limiter, fixed, sci , sci 10e, sci zerofill, precision=2},
            columns/L1_DFRG_error/.style={column name=DFRG, fixed, sci , sci 10e, sci zerofill, precision=2},
            columns/L2_DG_error/.style={column name=DG, fixed, sci , sci 10e, sci zerofill, precision=2},
            columns/L2_DG+limiter_error/.style={column name=$+$-limiter, fixed, sci , sci 10e, sci zerofill, precision=2},
            columns/L2_DFRG_error/.style={column name=DFRG, fixed, sci , sci 10e, sci zerofill, precision=2},
            columns/KL_DG+limiter_error/.style={column name=$+$-limiter, fixed, sci , sci 10e, sci zerofill, precision=2},
            columns/KL_DFRG_error/.style={column name=DFRG, fixed, sci , sci 10e, sci zerofill, precision=2},
            columns/date/.style={string type},
        ]{error_csv/average_error_norm_method_disc_bump_d_1_p_1_T_50-0_n_t_10_FR_weight_ref_true_is_exact_true/combined_error_values.csv}
    
        \caption{Error convergence for \cref{ex:1D-bump}, a 1D bump problem with polynomial order $p=1$, initial density equation \cref{eqn:bump_1D_density}, $b = 0.01$, $\mu = 0.5$, $k = 100$, and constant velocity $u(x) = 1$.
        The error is measured as the average of the $L^{1}$, $L^{2}$ norms, and $KL$ divergence with respect to the exact solution at points spaced out by $\Delta t = 0.01$ up to time $T=5$, for the DG, DG $+$-limiter and DFRG methods.}
        \label{table:bump_1D_error}

\end{table}

\vspace{-10pt}

\begin{figure}[h]
   \centering

   \includegraphics{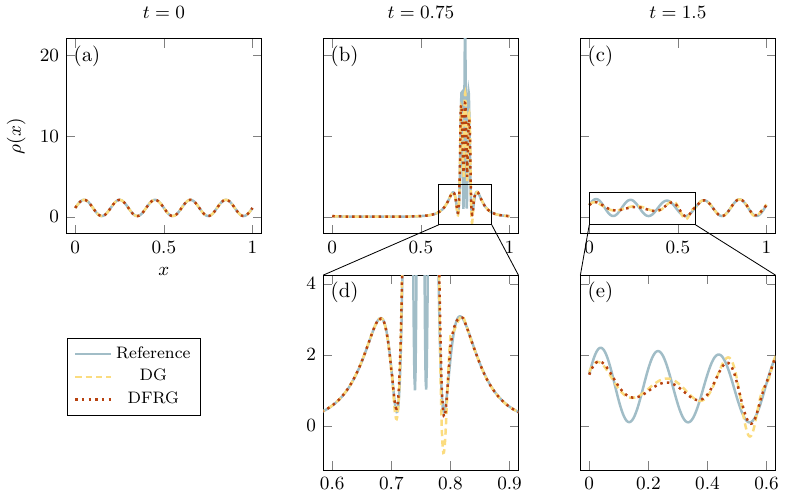}
   \includegraphics{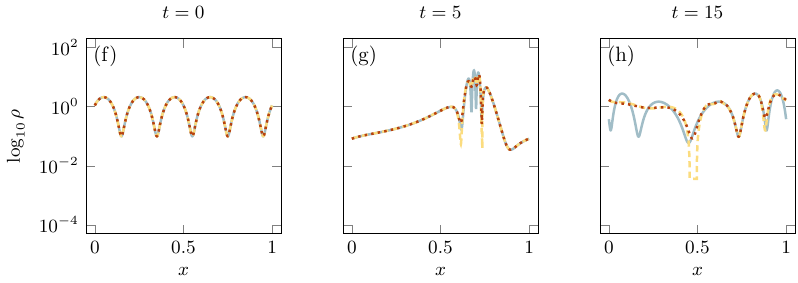}

   \caption{  Density profile for \cref{ex:1D-fine-details}, a 1D fine scale sinusoidal problem with polynomial order $p=3$, initial density $\rho_{0}(x) = sin(10 \pi x) + 1.1$, and fixed velocity $u(x) = sin(2 \pi x) + 1.2$.
   The solution is obtained using the DG and DFRG methods with CFL = 0.034375 and $m=64$ cells.
   Plots (f-h) are on a log scale.
   Profiles are shown at (a,f) $t=0$, (b,d,g) $t=0.75$, and (c,e,h) $t=1.5$.} 

   \label{fig:sinusoidal3_1D_density_log_combo}
   
\end{figure}

%\subsection*{Example 6.2.4: Resolving fine-details}
\begin{example}[Resolving fine-details]
\label{ex:1D-fine-details}
Higher-order schemes help resolve fine-details in the solution but can also introduce spurious oscillations near steep gradients, resulting in negative values.
This behavior is evident in the 1D fine-scale sinusoidal problem shown in \cref{fig:sinusoidal3_1D_density_log_combo}.
This example features a sinusoidal wave with fine-scale oscillations undergoing a compression and expansion cycle.
The initial density is $\rho(x,0) = \sin(10 \pi x) + 1.1$, and the velocity is $u(x) = \sin(2 \pi x) + 1.2$.
The symmetric velocity profile causes a symmetric compression and expansion cycle, so the fine-scale pattern initially observed reappears after a single cycle.
A higher-order Galerkin method resolves the fine-scale pattern with fewer cells than a lower-order method.
However, when using too few cells, neither method resolves the fine-scale pattern accurately in regions with higher velocity, though both perform well in regions of low velocity.
Under these conditions, the DG method produces negative values near the steep gradient, while the DFRG method maintains positivity.
\cref{fig:sinusoidal3_1D_p3_error_time_combo,table:sinusoidal3_1D_p3_error} shows the grid convergence of the DG and DFRG methods for the fine-scale sinusoidal problem.
As in the previous examples, the errors with respect to the $L^{1}$ and $L^{2}$ norms converge at approximately the same rate, while the KL divergence converges faster for the DFRG method.
When examining the error over time in \cref{fig:sinusoidal3_1D_p3_error_time_combo}, the error increases during compression periods and decreases during expansion periods.
\cref{fig:sinusoidal3_1D_density_log_combo} shows the log density plot for $m = 64$ elements.
Similar to the previous examples, the DFRG method achieves better accuracy in rarified regions than the DG method.
\end{example}

\begin{table}[H]
    \centering

        \pgfplotstableset{  col sep=comma,
                            columns={   Discretization,
                                        L1_DG_error,
                                        L1_DG+limiter_error,
                                        L1_DFRG_error,
                                        L2_DG_error,
                                        L2_DG+limiter_error,
                                        L2_DFRG_error,
                                        KL_DG+limiter_error,
                                        KL_DFRG_error
                                        },
                            font=\footnotesize
                        }
        \begin{filecontents*}{error_csv/average_error_norm_method_disc_sinusoidal3_d_1_p_3_T_15-0_n_t_1_FR_weight_ref_false_is_exact_true/combined_error_values.csv}
Discretization,L1_DG_error,L1_DG+limiter_error,L1_DFRG_error,L2_DG_error,L2_DG+limiter_error,L2_DFRG_error,KL_DG_error,KL_DG+limiter_error,KL_DFRG_error
16,0.5400840416679203,0.5345952509322176,0.5154156658010767,1.2460332308471358,1.2663891724525849,1.261725161766379,1.2625902168573357,0.7740734016247721,0.6856515014052869
32,0.4147444226647766,0.4176928828027452,0.3988523992172188,1.1443386767547359,1.1547779008924557,1.1606847615526883,0.6798166534939333,0.6545625896480796,0.6153866223459916
64,0.2566986376077198,0.25368959756031784,0.24007869791432573,0.8397439661399146,0.8511657098845624,0.8405257795607703,0.8047833574828828,0.557271099689713,0.41584824397056386
128,0.022265500892102687,0.022265500892102687,0.02373000984611855,0.09635813781633072,0.09635813781633072,0.09883897081507874,0.027025664871705095,0.027025664871705095,0.030442914112517618
256,0.0015169196409347888,0.0015169196409347888,0.0020648210702930137,0.004507145096355323,0.004507145096355323,0.007200543656342246,3.740324058394606e-5,3.740324058394606e-5,0.00036322768610994
\end{filecontents*}
\pgfplotstabletypeset[
            every head row/.style={
                before row= { \toprule
                & \multicolumn{3}{c}{L$^{1}$ error}
                & \multicolumn{3}{c}{L$^{2}$ error}
                & \multicolumn{2}{c}{KL error} \\
                },
                after row=\midrule},
            every last row/.style={
                after row=\bottomrule},
            columns/Discretization/.style={column name=$\#$ cells},
            columns/L1_DG_error/.style={column name=DG, fixed, sci , sci 10e, sci zerofill, precision=2},
            columns/L1_DG+limiter_error/.style={column name=$+$-limiter, fixed, sci , sci 10e, sci zerofill, precision=2},
            columns/L1_DFRG_error/.style={column name=DFRG, fixed, sci , sci 10e, sci zerofill, precision=2},
            columns/L2_DG_error/.style={column name=DG, fixed, sci , sci 10e, sci zerofill, precision=2},
            columns/L2_DG+limiter_error/.style={column name=$+$-limiter, fixed, sci , sci 10e, sci zerofill, precision=2},
            columns/L2_DFRG_error/.style={column name=DFRG, fixed, sci , sci 10e, sci zerofill, precision=2},
            columns/KL_DG+limiter_error/.style={column name=$+$-limiter, fixed, sci , sci 10e, sci zerofill, precision=2},
            columns/KL_DFRG_error/.style={column name=DFRG, fixed, sci , sci 10e, sci zerofill, precision=2},
            columns/date/.style={string type},
        ]{error_csv/average_error_norm_method_disc_sinusoidal3_d_1_p_3_T_15-0_n_t_1_FR_weight_ref_false_is_exact_true/combined_error_values.csv}
    
        \caption{Error convergence for \cref{ex:1D-fine-details}, a 1D sinusoidal problem with polynomial order $p=3$, initial density $\rho_{0}(x) = sin(10 \pi x) + 1.01$, and fixed velocity $u(x) = sin(2 \pi x) + 1.2$ . 
        The error is measured as the mean of the $L^{1}$, $L^{2}$ norms, and $KL$ divergence with respect to the exact solution at points spaced out by $\Delta t = 0.1$ up to time $T=15$ , for the DG, DG $+$-limiter, and DFRG methods.}
        \label{table:sinusoidal3_1D_p3_error}
\end{table}

\begin{figure}[H]

   \centering

   \includegraphics{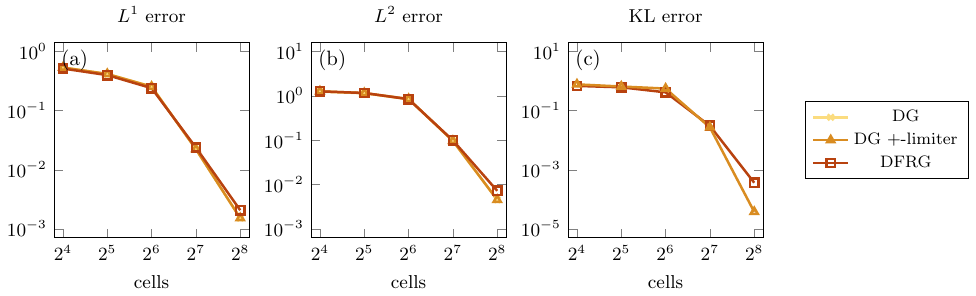}
   \includegraphics{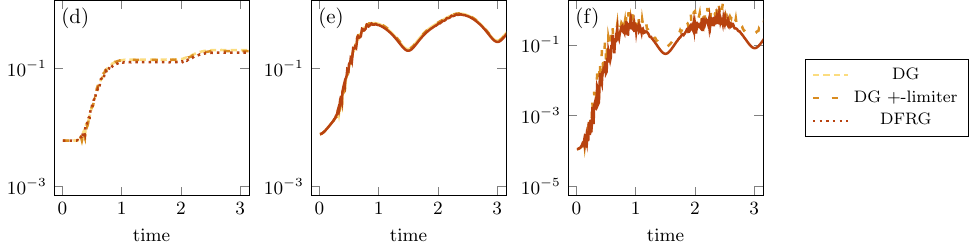}

   \caption{  Error convergence (a-c) and error over time (d-f) for \cref{ex:1D-fine-details}, a 1D sinusoidal problem with polynomial order $p=3$, initial density $\rho_{0}(x) = sin(10 \pi x) + 1.01$, and fixed velocity $u(x) = sin(2 \pi x) + 1.2$.
   The $L^{1}$, $L^{2}$ norms, and $KL$ divergence with respect to the exact solution at points spaced out by $\Delta t=0.1$ up to time $T=15$ is shown for the DG, DG $+$-limiter, and DFRG methods with CFL = 0.1875.
   The error convergence plots (a-c) show the average error over time, and the error over time plots (d-f) are for $m=256$ cells. }

   \label{fig:sinusoidal3_1D_p3_error_time_combo}

\end{figure}

\begin{figure}[h]
   \centering

   %\pgfplotsset{
   %   % this *defines* a custom colormap ...
   %       colormap={rustblue}{
   %           color(-0.05cm)=(rust);
   %           color(-0.005cm)=(lightlightsilver);
   %           color(0.5cm)=(steelblue);
   %           color(1.05cm)=(darksky)
   %       }
   %   }
  
      \scalebox{0.95}{
         \includegraphics{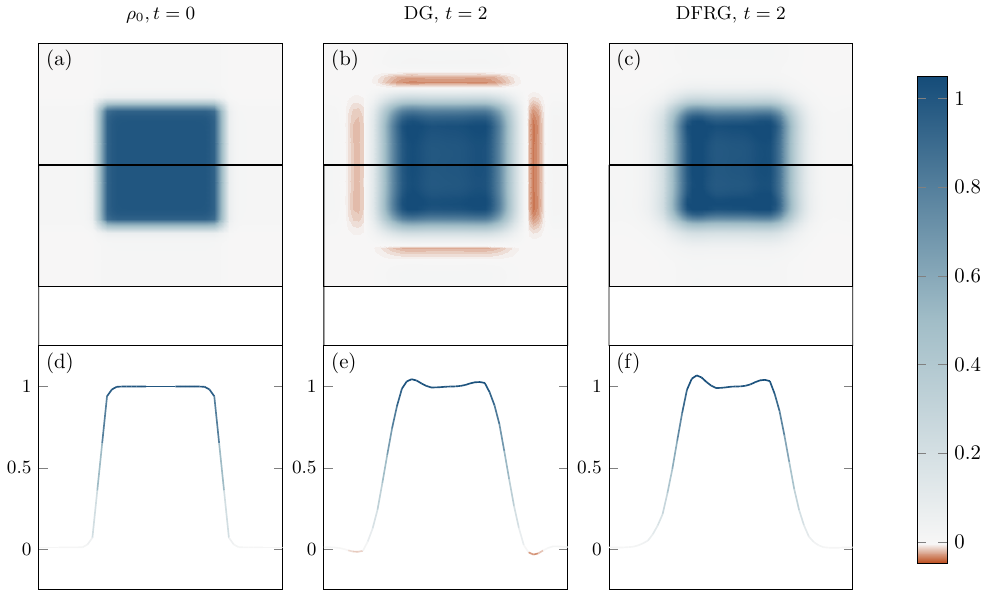}

      }

   \caption{ Density profile for \cref{ex:2D-bump}, a 2D bump problem with polynomial order $p=1$, initial density specified in Equation \cref{eqn:bump_2D_density}, and a constant velocity profile $\vct{u}(x,y) = \begin{bmatrix} 1 & 0.5 \end{bmatrix}$.
   The solution is obtained using the (b,e) DG and (c,f) DFRG methods with $m=32$ cells.
   Profiles are shown at (a,d) $t=0$, and (b-c,e-f) $t=2$.
   The plots (a-c) are contour plots of the density, while plots (d-f) show a slice of the density at $y=0.5$.} 

   \label{fig:bump_2D_density}

\end{figure}

\begin{table}[h]
    \centering

        \pgfplotstableset{  col sep=comma,
                            columns={   Discretization,
                                        L1_DG_error,
                                        L1_DG+limiter_error,
                                        L1_DFRG_error,
                                        L2_DG_error,
                                        L2_DG+limiter_error,
                                        L2_DFRG_error,
                                        KL_DG+limiter_error,
                                        KL_DFRG_error
                                        },
                            font=\footnotesize
                        }
        \begin{filecontents*}{error_csv/average_error_norm_method_disc_bump_d_2_p_1_1_T_3-0_n_t_1_FR_weight_ref_true_is_exact_true/combined_error_values.csv}
Discretization,L1_DG_error,L1_DG+limiter_error,L1_DFRG_error,L2_DG_error,L2_DG+limiter_error,L2_DFRG_error,KL_DG_error,KL_DG+limiter_error,KL_DFRG_error
8,0.1174123865761888,0.11695387448616208,0.13108400969309347,0.1705520047522213,0.1715190552855897,0.18604213355850957,14.057940248843694,13.90867666842642,8.718261497360512
16,0.06665946986647257,0.06174417729052383,0.07284967869901517,0.10937077165777369,0.10751730609930468,0.12248752453143116,18.482361040329895,29.031828877962077,8.321889583174169
32,0.0304324007961317,0.027184504214223592,0.03532532800241996,0.058385963822569796,0.056173051853316634,0.0703379554777593,20.965653847276684,31.79522176274068,5.750150779458696
64,0.010767787895558708,0.009634003594710217,0.01334119275885307,0.024586377993747396,0.023000353993559295,0.03180994342593489,17.75019043202978,27.63353279244315,2.354619288425562
\end{filecontents*}
\pgfplotstabletypeset[
            every head row/.style={
                before row= { \toprule
                & \multicolumn{3}{c}{L$^{1}$ error}
                & \multicolumn{3}{c}{L$^{2}$ error}
                & \multicolumn{2}{c}{KL error} \\
                },
                after row=\midrule},
            every last row/.style={
                after row=\bottomrule},
            columns/Discretization/.style={column name=$\#$ cells},
            columns/L1_DG_error/.style={column name=DG, fixed, sci , sci 10e, sci zerofill, precision=2},
            columns/L1_DG+limiter_error/.style={column name=$+$-limiter, fixed, sci , sci 10e, sci zerofill, precision=2},
            columns/L1_DFRG_error/.style={column name=DFRG, fixed, sci , sci 10e, sci zerofill, precision=2},
            columns/L2_DG_error/.style={column name=DG, fixed, sci , sci 10e, sci zerofill, precision=2},
            columns/L2_DG+limiter_error/.style={column name=$+$-limiter, fixed, sci , sci 10e, sci zerofill, precision=2},
            columns/L2_DFRG_error/.style={column name=DFRG, fixed, sci , sci 10e, sci zerofill, precision=2},
            columns/KL_DG+limiter_error/.style={column name=$+$-limiter, fixed, sci , sci 10e, sci zerofill, precision=2},
            columns/KL_DFRG_error/.style={column name=DFRG, fixed, sci , sci 10e, sci zerofill, precision=2},
            columns/date/.style={string type},
        ]{error_csv/average_error_norm_method_disc_bump_d_2_p_1_1_T_3-0_n_t_1_FR_weight_ref_true_is_exact_true/combined_error_values.csv}
    
        \caption{Error convergence for \cref{ex:2D-bump}, a 2D bump problem with polynomial order $p=1$, initial density equation \cref{eqn:bump_2D_density}, and fixed velocity $\vct{u}(x,y) = \begin{bmatrix} 1 \\ 0.5 \end{bmatrix} $.
         The mean of the $L^{1}$, $L^{2}$, and $KL$ errors with respect to the exact solution is measured at points spaced out by $\Delta t = 0.01$ up to time $T=3$, for the DG, DG $+$-limiter, and DFRG methods.}
        \label{table:bump_2D_error}

\end{table}

%\subsection*{Example 6.2.5: 2D Bump advection}
\begin{example}[2D Bump advection]
\label{ex:2D-bump}
Many behaviors observed in the 1D cases also appear in the 2D case.
A representative example is the 2D bump advection problem shown in \cref{fig:bump_2D_density}.
This example uses a constant velocity profile $\vct{u}(x,y) = \begin{bmatrix} 1  &  0.5 \end{bmatrix}$, with the initial density profile
\begin{equation}
   \rho^{(2)}(x,y,0) = \rho(x,0) \rho(y,0), 
   \label{eqn:bump_2D_density}
\end{equation}
where $\rho$ is the 1D bump advection problem defined in \cref{eqn:bump_1D_density}.

\begin{figure}[h]
   \centering
   
   \scalebox{0.95}{
      \includegraphics{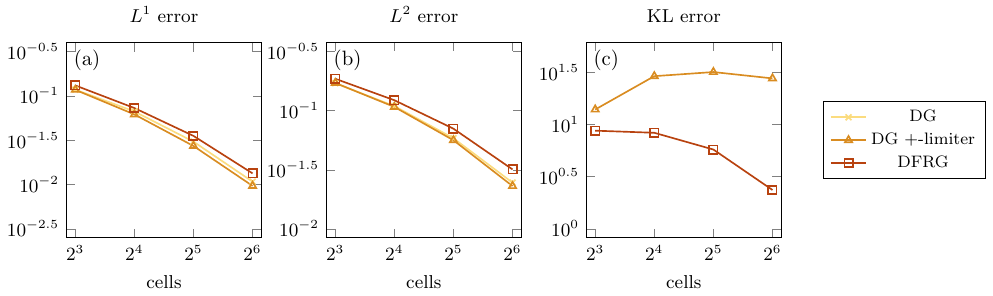}
   }

   \caption{ Error convergence for \cref{ex:2D-bump}, a 2D bump problem with polynomial order $p=1$, initial density equation \cref{eqn:bump_2D_density}, and fixed velocity $\vct{u}(x,y) = \begin{bmatrix} 1 & 0.5 \end{bmatrix} $.
   The mean of the (a) $L^{1}$, (b) $L^{2}$, and (c) $KL$ errors with respect to the exact solution is measured at points spaced out by $\Delta t=0.01$ up to time $T=3$, for the DG, DG $+$-limiter, and DFRG methods.}

   \label{fig:bump_2D_error}

\end{figure}

Similar to the 1D case, oscillations near the steep gradient cause the DG method to produce negative values, while the DFRG method preserves positivity throughout.
\cref{fig:bump_2D_error,table:bump_2D_error} shows the error convergence of the DG and DFRG methods for the 2D bump advection problem.
The errors with respect to the $L^{1}$ and $L^{2}$ norms converge at approximately the same rate for both methods, although the DFRG method performs slightly worse on these metrics.
However, the KL divergence converges much faster for the DFRG method than for the DG method with the ($+$)-limiter.
\end{example}

\begin{figure}[H]
   \centering

   %\pgfplotsset{
   %% this *defines* a custom colormap ...
   %   colormap={rusttosteel}{
   %      color(-0.05cm)=(rust);
   %      color(-0.005cm)=(lightlightsilver);
   %      color(0.5cm)=(steelblue);
   %      color(1.05cm)=(darksky)
   %   }
   %}

   \scalebox{0.94}{
      \includegraphics{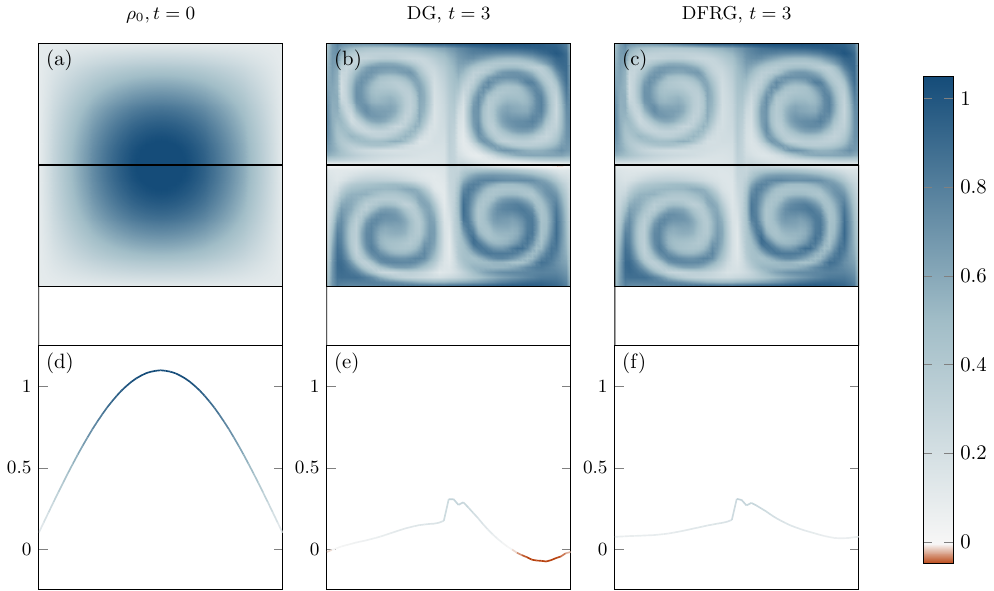}
   }

   \caption{ Density profile for \cref{ex:2D-swirl}, the 2D swirl problem with polynomial order $p=1$, initial density $\rho(x,y,0) = \sin( \pi x ) \sin( \pi y ) + 0.1$, and a constant velocity profile specified in equation \cref{eqn:swirl2_2D_velocity}.
   The solution is obtained using the (b,e) DG and (c,f) DFRG methods with $m=32$ cells.
   Profiles are shown at (a,d) $t=0$ and (b-c,e-f) $t=3$.
   The plots (a-c) are contour plots of the density, while plots (d-f) show a slice of the density at $y=0.5$.} 
      
   \label{fig:swirl2_2D_density}
\end{figure}

%\subsection*{Example 6.2.6: Resolving fine-details in 2D}
\begin{example}[Resolving fine-details in 2D]
\label{ex:2D-swirl}
Positivity-preserving methods often raise concerns about introducing excess numerical dissipation, which can smooth out fine details in the solution.
We demonstrate that the DFRG method resolves fine details in the solution without producing negative values.
The 2D swirl problem, shown in \cref{fig:swirl2_2D_density}, features a velocity field that generates a swirling motion.
The velocity profile is 
\begin{equation}
   \begin{split}
      u_{y}(x,y) = \sin( 2 \pi x ) \sin( 2 \pi ( y - 0.25 ) )  + 0.1, \\
      u_{x}(x,y) = \cos( 2 \pi x ) \cos( 2 \pi ( y - 0.25) ) + 0.1,
   \end{split}
   \label{eqn:swirl2_2D_velocity}
\end{equation}
and the initial density is 
\begin{equation}
   \rho(x,y,0) = \sin( \pi x ) \sin( \pi y ) + 0.1.
\end{equation}

As mass moves inward, the DG solution produces negative values on the outer edges of the swirl pattern.
In contrast, the DFRG solution avoids these negative values while preserving the swirl pattern's fine details.
\end{example}

\begin{figure}[h]
   \centering

   \scalebox{0.80}{
      \includegraphics{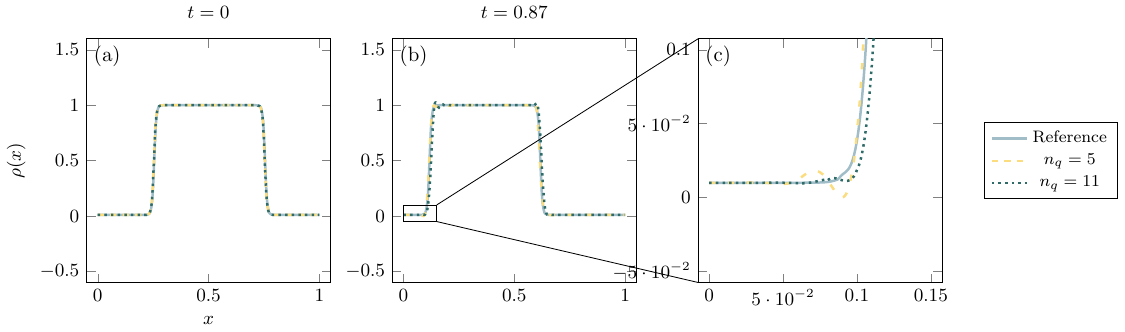}
   }

   \caption{ Density profile for \cref{ex:1D-quadrature-failure-mode}, a 1D bump problem with polynomial order $p=1$, initial density equation \cref{eqn:bump_1D_density}, $b = 0.01$, $\mu = 0.25$, $k = 200$, and constant velocity $u(x) = 1$.
   The solution is obtained using the DFRG method with $n_q = 5$ and $n_q = 11$ quadrature points,  polynomial order $p=3$, and $m=50$ cells.
   The profiles are shown at (a) $t=0$ and (b-c) $t=0.87$, plot (c) at $t=0.87$  is zoomed in to show the failure point near the steep gradient.} 
   \label{fig:bump_failed_1_1D_p3_density}
   
\end{figure}

%\subsection*{Example 6.2.7: Common failure mode}
\subsection*{Common failure mode}
The positivity-preserving property of the DFRG method depends on the choice of time discretization and quadrature.
The semidiscretization is derived in the continuous time limit. 
For finite time steps, positivity preservation is not guaranteed.
Satisfying the CFL condition mitigates this issue to some extent, but steep gradients within a single cell, especially with small densities on one end, can still cause failures.
In such cases, the Fisher-Rao weight may become orders of magnitude larger on one end of the cell compared to the rest of the domain, requiring very small time steps to maintain positivity.

\begin{example}[Failure due to quadrature]
\label{ex:1D-quadrature-failure-mode}
Quadrature accuracy plays a critical role in computing inner products.
A lack of quadrature points introduces quadrature error, increasing the likelihood of failure.
Conversely, improving quadrature accuracy reduces this risk.
\cref{fig:bump_failed_1_1D_p3_density} illustrates the solution to the 1D bump advection problem with $k = 0.005$, $\mu = 0.25$, and $b = 0.01$, which results in a steep density gradient.
Using third-order DFRG, the method fails with $n_q = 5$ quadrature points but succeeds with $n_q = 11$ quadrature points.
\end{example}

\begin{figure}[h]
   \centering

   \scalebox{0.8}{
         \includegraphics{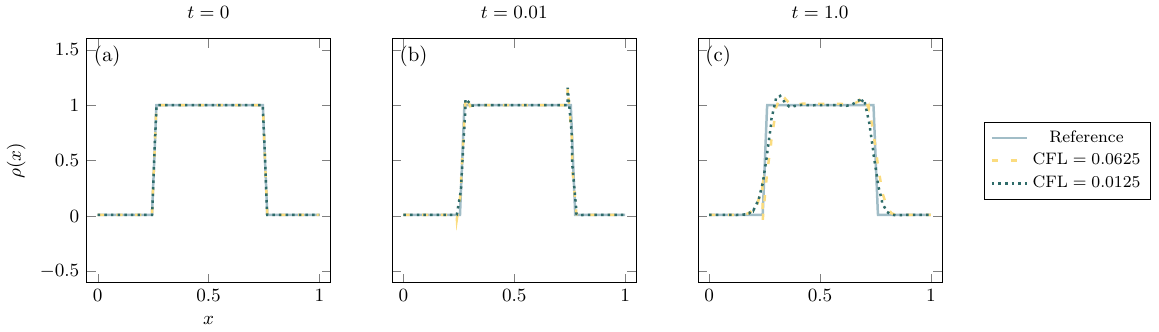}
      }

      \caption{ Density profile for \cref{ex:1D-time-step-failure-mode}, a 1D bump problem with polynomial order $p=1$, initial density equation \cref{eqn:bump_1D_density}, $b = 0.01$, $\mu = 0.25$, $k = 1000$, and constant velocity $u(x) = 1$.
      The solution is obtained using the DFRG method with $n_q = 5$ quadrature points,  polynomial order $p=1$, and $m=50$ cells.
      The profiles are shown at (a) $t=0$, (b) $t=0.01$, and (c) $t=1$.
      }

      \label{fig:bump_failed_2_1D_density}
   
\end{figure}

\vspace{-10pt}

\begin{example}[Failure due to time step]
\label{ex:1D-time-step-failure-mode}
Reducing the time step can minimize the chance of failure by limiting the size of negative density updates caused by imprecise inner products.
\cref{fig:bump_failed_2_1D_density} shows the solution to the 1D bump advection problem with $k = 0.001$, $\mu = 0.25$, and $b = 0.01$, which features an even steeper gradient than the previous example.
The method fails when CFL = 0.0625 but succeeds when the CFL is reduced to 0.0125.
\end{example}

Increasing the number of spatial cells in regions with steep gradients can also prevent failure.
Combining all these strategies adaptively in scenarios prone to failure can enhance the reliability of the DFRG method without significantly increasing computational costs.

\section{Comparison, conclusion, and outlook}
\subsection*{Comparison to prior work}
Maximum likelihood estimation is closely related to the maximum entropy principle.
In exponential families, it is equivalent to combining the latter with suitable moment constraints \cite{wainwright2008graphical}.
This relates our work to existing approaches employing entropy regularization for abstract or physical problems \cite{cuturi2013sinkhorn,lindsey2023fast,keith2024proximal,dokken2025latent}.
Different from existing works, we address dynamic hyperbolic problems that lack minimization or gradient flow structure.
Another related line of work are maximum entropy closures of kinetic problems \cite{levermore1996moment,hauck2008convex,porteous2021data}.
Different from these works, our approach applies not in the momentum, but the spatial dimension.
Furthermore, our work provides an alternative, and thus relates to, limiter-based approaches for positivity preservation \cite{shu2018bound}.
This makes its combination with PDE-based approaches to shock capturing, such as \cite{cao2023information,cao2024information}, especially promising.
More broadly, it is part of a growing research field investigating the interaction between transport and information geometric structures on families of probability measures \cite{jordan1998variational,amari2018information,amari2019information,peyre2019computational,cui2022time,li2023wasserstein,amari2023information}.
Our work views $\rho$ as a statistical estimator of physical reality, rather than a physical object itself, and adapts the Galerkin projection accordingly.
It thus serves as another example of ``information geometric mechanics'' in the sense of \cite{cao2023information}.

\subsection*{Conclusion and outlook}
In this work, we observe that the Galerkin projection of positivity constrained problems amounts to performing statistical inference using the method of moments (MoM).
Loss of positivity thus arises from a well-known limitation of MoM--- its susceptibility to producing estimates inconsistent with the observed data.
Our novel \emph{maximum likelihood discretization} overcomes this problem by replacing MoM with maximum likelihood estimation.
In the time-continuous limit, it simplifies to the Fisher-Rao Galerkin semidiscretization (FRG).
We show empirically that FRG preserves positivity and prove that it yields error bounds in the Kullback-Leibler divergence.
To simplify the derivation, our current work is limited to the linear transport equation.
Ongoing and future work includes its extension to other conservation laws, such as the Euler, Navier-Stokes, MHD, Boltzman and Vlasov systems.
Replacing the Kullback-Leibler divergence with other Bregman divergences promises extensions to a vast range of constraints.
We intentionally made this work agnostic to the particular choice of trial and test spaces.
In the future, we will investigate particular choices, such as local polynomials, to derive explicit error bounds and rigorously prove the positivity preservation observed in our numerical experiments.

\section*{Acknowledgments}
The authors gratefully acknowledge support from the Air Force Office of Scientific Research under award number FA9550-23-1-0668 (Information Geometric Regularization for Simulation and Optimization of Supersonic Flow) and from the Office of Naval Research under award number N00014-23-1-2545 (Untangling Computation)

\bibliography{references}
\bibliographystyle{plain}

\end{document}